\documentclass[11pt, a4paper]{amsart}
\usepackage{amscd,amssymb,eucal,eufrak,mathrsfs,amsmath}
\usepackage{hyperref} 

\input xy
\xyoption{all}
\usepackage{epsfig}

\newtheorem{thm}{Theorem}[section]
\newtheorem{cor}[thm]{Corollary}
\newtheorem{lem}[thm]{Lemma}
\newtheorem{prop}[thm]{Proposition}
\newtheorem{fact}[thm]{Fact}

\theoremstyle{definition}
\newtheorem{defn}[thm]{Definition}

\theoremstyle{remark}
\newtheorem{remark}[thm]{Remark}
\newtheorem{remarks}[thm]{Remarks}
\newtheorem{example}[thm]{Example}

\numberwithin{equation}{section}

\newcommand{\delete}[1]{} 
\newcommand{\nt}{\noindent}

\def\eps{{\varepsilon}}
\newcommand{\ep}{\varepsilon}
\newcommand{\sk}{\vskip 0.2cm}

\newcommand{\nl}{\newline}
\newcommand{\ben}{\begin{enumerate}}

\newcommand{\een}{\end{enumerate}}
\newcommand{\bit}{\begin{itemize}}

\newcommand{\eit}{\end{itemize}}

\newcommand{\rest}{\upharpoonright}

\def\R {{\mathbb R}}
\def\N {{\mathbb N}}
\def\Z {{\mathbb Z}}

\def\calB{\mathcal B}

\def\norm#1{\left\Vert#1\right\Vert}

\def\Iso{{\mathrm{Iso}}\,}

\def\diam{{\mathrm{diam}}}

\def\QED{\nobreak\quad\ifmmode\roman{Q.E.D.}\else{\rm Q.E.D.}\fi}

\def\sbs{\subset}
\def\a{\alpha}

\def\g{\gamma}

\newcommand{\br}{\vspace{3 mm}}

\newcommand{\cls}{{\rm{cls\,}}}

\newcommand{\Acal}{\mathcal{A}}

\newcommand{\co}{\rm{co}}
\newcommand{\card}{\rm{card\,}}

\begin{document}

\title[]
{Representations of dynamical systems on Banach spaces
not containing $l_1$}

\author[]{E. Glasner}
\address{Department of Mathematics, Tel-Aviv University,
Tel Aviv, Israel} \email{glasner@math.tau.ac.il}
\urladdr{http://www.math.tau.ac.il/$^\sim$glasner}

\author[]{M. Megrelishvili}
\address{Department of Mathematics,
Bar-Ilan University, 52900 Ramat-Gan, Israel}
\email{megereli@math.biu.ac.il}
\urladdr{http://www.math.biu.ac.il/$^\sim$megereli}

\date{May 17, 2012}

\keywords{Baire one function, Banach representation of
dynamical systems,
enveloping semigroup,
fragmentability,
Rosenthal's dichotomy, Rosenthal's compact, Tame system}

\begin{abstract}
For a topological group $G$,
we show that a compact metric $G$-space is tame if and only if it
can be linearly represented on a separable Banach space which does
not contain an isomorphic copy of $l_1$
(we call such Banach spaces, Rosenthal spaces).
With this goal in mind we study tame dynamical systems
and their representations on Banach spaces.
\end{abstract}

\thanks{Research partially supported by BSF (Binational USA-Israel)
grant no. 2006119.}

\thanks{{\em 2000 Mathematical Subject Classification} 37Bxx, 54H20,
54H15, 46xx}

\maketitle
\tableofcontents

\section{Introduction}

\subsection{Some important dichotomies}

Rosenthal's celebrated dichotomy theorem asserts that every bounded
sequence in a Banach space either has a weak Cauchy subsequence or
admits
a subsequence equivalent to the unit vector basis of $l_1$ (an
$l_1$-\emph{sequence}).
Thus, a Banach space $V$ does not contain an  $l_1$-sequence
(equivalently, does not contain an isomorphic copy of $l_1$) if
and only if every bounded sequence in $V$ has a weak-Cauchy
subsequence \cite{Ros0}. 
In the present work we
call a Banach space satisfying these
equivalent conditions a {\em Rosenthal space}.

The theory of Rosenthal spaces is one of the cases where the
interplay between Analysis and Topology gives rise to many deep
results. Our aim is to show the relevance of Topological Dynamics
in this interplay. In particular, we examine
when a dynamical system can be represented on a Rosenthal space,
and show that being {\em tame} is a complete characterization of
such systems.

First we recall some results and ideas. The following
dichotomy is a version of a result of Bourgain, Fremlin and Talagrand
\cite{BFT} (as presented in the book of Todor\u{c}evi\'{c}
\cite{To-b}, Proposition 1 of Section 13).

\begin{fact} [BFT dichotomy]\label{f:BFT}
Let $X$ be a Polish space and
let $\{f_n\}_{n=1}^\infty \subset C(X)$ be a sequence of real
valued
functions which is {\em pointwise bounded}
(i.e. for each $x\in X$ the sequence $\{f_n(x)\}_{n=1}^\infty$ is bounded in $\R$).
Let $K$ be the pointwise closure of $\{f_n\}_{n=1}^\infty$
in $\R^X$. Then either $K \subset \calB_1(X)$, where $ \calB_1(X)$
denotes the space of all real valued Baire 1 functions on $X$,
or $K$ contains a homeomorphic copy of $\beta\N$.
\end{fact}

In \cite[Theorem 3.2]{GM} the following dynamical dichotomy, in the
spirit of Bourgain-Fremlin-Talagrand theorem, was established.

\begin{fact}[A dynamical BFT dichotomy]\label{f:D-BFT}
Let $X$ be a
compact
metric dynamical $G$-system and let $E(X)$ be its
enveloping semigroup. We have the following dichotomy. Either
\begin{enumerate}
\item
$E(X)$ is a separable Rosenthal compactum, hence with cardinality \nl
${\card}{E(X)} \leq 2^{\aleph_0}$; \ or
\item
$E(X)$
contains a homeomorphic copy of $\beta\N$,
hence ${\card}{E(X)} = 2^{2^{\aleph_0}}$.
\end{enumerate}
\end{fact}

In \cite{Gl-tame} a compact metric dynamical system is called {\em tame}
if the first alternative occurs, i.e. $E(X)$ is a Rosenthal
compactum. By \cite{BFT} every Rosenthal compactum is a Fr\'echet
space (and in particular its topology is determined by the
converging sequences). Thus, either $E(X)$ (although not
necessarily metrizable) has a nice topological structure, or it is
as unruly as possible containing a copy of $\beta\N$. As to the
metrizability of $E(X)$, recent results \cite{GM} and \cite{GMU}
assert that $E(X)$ is metrizable iff the metric compact
$G$-space $X$ is hereditarily non-sensitive (HNS), iff
$X$ is Asplund representable (see Section \ref{subs:RN}).
A Banach space $V$ is an {\em Asplund\/} space if
the dual of every separable Banach subspace is separable (see Remarks \ref{r:fr}.4).
Reflexive spaces and spaces of the type $c_0(\Gamma)$
are Asplund.

\subsection{The main results and related facts}

The main result of the present work is the following:

\begin{thm} \label{t:main-int}
Let $X$ be a compact metric $G$-space.
The following conditions are equivalent:
\begin{enumerate}
\item $(G,X)$ is a tame $G$-system.
\item $(G,X)$ is representable on a separable Rosenthal
Banach space.
\end{enumerate}
\end{thm}

This theorem continues a series of recent results which link
dynamical properties of $G$-systems (like WAP and HNS) to their
representability  on ``good" Banach spaces
(Reflexive and Asplund respectively). See Sections
\ref{s:BanHier} and \ref{s:Ban Rep} below for more details.

One of the important questions in Banach space theory until the mid 70's
was how to construct a separable Rosenthal space which is not Asplund.
The first examples were constructed independently by
James \cite{Ja} and Lindenstrauss and Stegall \cite{LiSt}.
In view of Theorem \ref{t:main-int} we now see that a fruitful
way of producing such distinguishing examples comes from dynamical
systems. Just consider a compact metric tame $G$-system which is
not HNS (see e.g.
Remarks \ref{r:old} below)
and then apply Theorem \ref{t:main-int}.

In order to get a better perspective on the position of
tame systems in the hierarchy of dynamical systems
we remind the reader of some enveloping semigroup characterizations.
For a recent review of enveloping semigroup theory we refer to
\cite{G4}.
A compact $G$-space $X$ is WAP
(weakly almost periodic) if and only if its enveloping semigroup
$E(X) \subset X^X$ consists of continuous maps (Ellis and Nerurkar \cite{EN}).
Recently the following characterization of tameness was
established.

\begin{fact} \label{GMUint} \cite{GMU}
A compact metric dynamical $G$-system $X$ is tame if and only if
every element of $E(X)$ is a Baire
class 1 function (equivalently, has the point of continuity property)
from $X$ to itself.
\end{fact}

A function $f: X \to Y$ has
the {\em point of continuity property}
if for every closed nonempty subset $A$ of $X$ the restriction
$f|_A: A \to Y$ has a point of continuity. For compact $X$ and
metrizable $Y$ it is equivalent to the \emph{fragmentability} (see
Section \ref{sec:fr} and Lemma \ref{Baire-1}) of the function $f$.
The topological concept of fragmentability comes in fact from
Banach space theory (Jayne-Rogers \cite{JR}). For dynamical applications of
fragmentability we refer to
\cite{Mefr, Meop, Menz, GM, GM-suc, GM-fixed}.

Fact \ref{GMUint} suggests the following general definition.

\begin{defn} \label{d:TAME}
 Let $X$ be a (not necessarily metrizable) compact $G$-space.
 We say that $X$ is
\emph{tame} if for every element $p \in E(X)$ the function $p: X
\to X$ is fragmented.
\end{defn}

The class of tame dynamical systems contains the class of HNS
systems and hence also WAP systems.
Indeed, as we already mentioned, every
function $p: X \to X$ ($p \in E(X)$) is continuous for WAP
systems. As to the HNS systems they can be characterized as those
$G$-systems where the {\em family} of maps
$\{p: X \to X\}_{p \in E(X)}$
is a \emph{fragmented family} (see Definition
\ref{d:fr-family} and Fact \ref{t:HNS} below). In particular, every individual $p: X
\to X$ is a fragmented map. Thus, these enveloping semigroup
characterizations yield a natural hierarchy of the three
classes, WAP, HNS and Tame, dynamical systems.

In \cite{Ko} K\"{o}hler introduced
the definition of regularity for cascades (i.e. $\Z$-dynamical
systems) in terms of
\emph{independent} sequences and,
using results of Bourgain-Fremlin-Talagrand has shown that
her definition can be reformulated in terms of $l_1$-sequences.
Extending K\"{o}hler's definition to arbitrary topological groups
$G$ we say that compact $G$-space $X$ is \emph{regular} if,
for any $f \in C(X)$, the orbit $fG$ does not contain an
$l_1$-sequence (in other words the second
alternative is ruled out in the Rosenthal dichotomy).
As we will see later, in Corollary \ref{c:tame-char}, a $G$-system
is regular if and only if it is tame (for metrizable $X$ this fact
was
established
in \cite{Gl-tame}).

In Theorem \ref{t:X-repres-iff} we give a characterization of
Rosenthal representable $G$-systems. As a particular case (for
trivial $G$) we get a topological characterization of compact
spaces which are homeomorphic to weak$^*$ compact subsets in the
dual of Rosenthal spaces. A well known result characterizes
Rosenthal spaces as those Banach spaces whose dual has the weak
Radon-Nikod\'ym property \cite[Corollary 7.3.8]{Tal}. It is
therefore natural to call such a compact space a \emph{weakly
Radon-Nikod\'ym} compactum (WRN). Theorem \ref{t:WRN} gives a
simple characterization in terms of fragmentability.
Namely,
{\em a compact space $X$ is WRN iff
there exists a
bounded subset $F \subset C(X)$ such that the pointwise closure of
$F$ in $\R^X$ consists of fragmented maps
from $X$ to $\R$
and $F$ separates points of $X$}.

Theorem \ref{t:X-repres-iff} is
related to yet another
characterization of Rosenthal Banach spaces. Precisely,
let $V$ be a Banach space with dual $V^*$ and second dual
$V^{**}$. One may
consider the elements of $V^{**}$ as functions on the
weak$^*$ compact unit ball $B^*:=B_{V^*} \subset V^*$. While the
elements of $V$ are clearly continuous on $B^*$ it is not true in
general for elements from $V^{**}$. By a result of Odell and
Rosenthal \cite{OR}, a \emph{separable} Banach space $V$ is
Rosenthal iff every element $v^{**}$ from $V^{**}$ is a Baire one
function on $B^*$. More generally  E. Saab and P. Saab \cite{SS}
show that $V$ is Rosenthal iff every element of $V^{**}$ has
the point of continuity property when restricted to $B^*$.
Equivalently, every restriction of $v^{**}$ to a bounded subset
$M$ is fragmented as a function $(M, w^*) \to \R$ (see Fact
\ref{t:SS} below).

Answering a question of Talagrand \cite[Problem14-2-41]{Tal},
R. Pol \cite{Pol} gave an example of a separable
compact Rosenthal space $K$ which cannot be
embedded in $\calB_1(X)$ for any compact metrizable $X$.
We say that a compact space $K$ is \emph{strongly Rosenthal} if
it is homeomorphic to a subspace of $\calB_1(X)$ for a compact
metrizable $X$.
We say that a compact space $K$ is
\emph{admissible}
if
there exists a metrizable compact space $X$ and a bounded subset $Z
\subset C(X)$ such that the pointwise closure $\cls_p(Z)$ of $Z$
in $\R^X$ consists of Baire 1 functions and $K \subset \cls_p(Z)$.
Clearly every admissible compactum is strongly Rosenthal.
We do not know whether these two classes of compact spaces coincide.
Note that the enveloping semigroup $K:=E(X)$ of a compact
metrizable $G$-space $X$ is admissible iff
$(G,X)$ is tame (Proposition \ref{p:EisTame}).

As another consequence of our analysis we show that {\em a compact space
$K$ is an admissible Rosenthal compactum iff it is homeomorphic to
a weak$^*$ closed bounded subset in the
second dual of a separable Rosenthal Banach space $V$}
(Theorem \ref{t:tame-comp}).

\begin{remark}
We note that the main results of our work remain true for
semigroup actions once some easy modifications are introduced.
\end{remark}

\begin{remark}
The attentive reader will not fail to detect the major importance
to our work of the papers \cite{DFJP}, \cite{BFT}, and
the book \cite{Tal}.
\end{remark}

\sk
\subsection{The hierarchy of Banach representations}
\label{s:BanHier}

In the following table we encapsulate some features of the
trinity: dynamical systems, enveloping semigroups, and Banach
representations. Let $X$ be a compact metrizable $G$-space and
$E(X)$ denote the corresponding enveloping semigroup. The symbol
$f$ stands for an arbitrary function in $C(X)$ and
$fG = \{f \circ g: g\in G\}$ denotes its orbit.
Finally, $\cls(fG)$ is the pointwise closure of $fG$ in
$\R^X$.


{\tiny{
\begin{table}[h]
\begin{center}
\begin{tabular}{ | l | l| l | l | l | l | }
\hline  & Dynamical characterization &  Enveloping semigroup &
Banach representation\\
\hline  WAP & $\cls(fG)$ is a subset of $C(X)$  &
Every element is continuous & Reflexive \\
\hline HNS & $\cls(fG)$ is metrizable & $E(X)$ is metrizable
& Asplund\\
\hline Tame &  $\cls(fG)$ is Fr\'echet &
Every element is Baire 1 & Rosenthal\\
\hline
\end{tabular}

\sk
\caption{ \protect  The hierarchy of Banach representations}
\end{center}
\end{table}
}}


\section{Topological background: fragmentability and Baire 1 functions}
\label{sec:fr}

Let $X$ be a topological space and $A \subset X$. We say that $A$
is {\em relatively compact} in $X$ if the closure ${\cls}(A)$ is a
compact subset of $X$. We say that $A$ is {\em sequentially precompact}
in $X$ if every sequence in $A$ has a subsequence which converges
in $X$. Compact space will mean compact and Hausdorff.

The following definition is a generalized version of
{\it fragmentability}.

\begin{defn} \label{def:fr}
\cite{JOPV, Mefr} Let $(X,\tau)$ be a topological space and
$(Y,\mu)$ a uniform space. We say that $X$ is {\em $(\tau,
\mu)$-fragmented\/} by a
(typically not continuous)
function $f: X \to Y$ if for every nonempty subset $A$ of $X$ and every entourage $\eps
\in \mu$ there exists an open subset $O$ of $X$ such that $O \cap
A$ is nonempty and the set $f(O \cap A)$ is $\eps$-small in $Y$.
We also say in that case that the function $f$ is {\em
fragmented\/}. Notation: $f \in {\mathcal F}(X,Y)$, whenever the
uniformity $\mu$ is understood. If $Y=\R$ then we write simply
${\mathcal F}(X)$.
\end{defn}

\begin{remarks} \label{r:fr}
\ben
\item
In Definition \ref{def:fr}.1 when $Y=X, f={id}_X$ and $\mu$ is a
metric uniform structure, we get the usual definition of
fragmentability in the sense of Jayne and Rogers \cite{JR}.
Implicitly it already appears in a paper of Namioka and Phelps
\cite{NP}.

\item
It is enough to check the condition of Definition \ref{def:fr}
only for closed subsets $A \subset X$ and
for $\ep \in \mu$ from a {\it subbase} $\gamma$ of $\mu$ (that is,
the finite intersections of the elements of $\gamma$ form a base
of the uniform structure $\mu$).

\item
Namioka's {\it joint continuity theorem} \cite{N-jct} implies that
every weakly compact subset $K$ of a Banach space is
(weak,norm)-fragmented (that is, $id_K: (K,weak) \to (K,norm)$ is
fragmented).
\item
Recall that a Banach space $V$ is an {\em Asplund\/} space if the
dual of every separable Banach subspace is separable, iff every
bounded subset $A$ of the dual $V^*$ is
(weak${}^*$,norm)-fragmented, iff $V^*$ has the Radon-Nikod\'ym
property.
Reflexive spaces and spaces of the type $c_0(\Gamma)$
are Asplund. For more details cf. \cite{Bo, Fa, N}.


\item A Banach space $V$ is Rosenthal if and only if every bounded
subset $A$ of the dual $V^*$ is (weak${}^*$ topology, weak
uniformity)-fragmented. This follows by Proposition \ref{p:crit}.
\een
\end{remarks}

Recall that $f: X \to Y$ is {\em barely continuous\/}, \cite{MN},
if for every nonempty closed subset $A \subset X$, the restricted
map $f\rest_A$ has at least one point of continuity. Following
\cite[Section 14]{Tal} the set of barely continuous functions
$f: X \to \R$ is denoted by $B_r'(X)$.

\begin{lem} \label{simple-fr}
\begin{enumerate}


\item
Every barely continuous
$f$
is fragmented.

\item
Let $\a: X \to Y$ be a continuous map. If $f: Y \to (Z,\mu)$ is a
fragmented map then the composition $f \circ \a: X \to (Z,\mu)$ is
also fragmented.

\item
Let $p: X \to Y$ be a map from a topological space $X$ into a
compact space $Y$. Suppose that $\{f_i: Y \to Z_i\}_{i \in I}$ is
a system of continuous maps from $Y$ into Hausdorff uniform spaces
$Z_i$ such that it separates points of $Y$ and $f_i \circ p \in
{\mathcal F}(X,Z_i)$ for every $i \in I$. Then $p \in {\mathcal F}(X,Y)$.


\item \label{l-quot-fr}
Let $(X,\tau)$ and $(X',\tau')$ be compact spaces, and let
$(Y, \mu)$ and $(Y', \mu')$ be uniform spaces. Suppose that:
$\a: X \to X'$ is a continuous onto map, $\nu: (Y, \mu) \to
(Y', \mu')$ is uniformly continuous, $\phi: X \to Y$ and
$\phi': X' \to Y'$ are maps such that the following diagram

\begin{equation*}
\xymatrix {
(X, \tau) \ar[d]_{\a} \ar[r]^{\phi} & (Y, \mu)
\ar[d]^{\nu} \\
(X', \tau') \ar[r]^{\phi'} & (Y', \mu') }
\end{equation*}
commutes. If $X$ is fragmented by $\phi$ then $X'$ is
fragmented by $\phi'$.

\item \label{l:factor}
Let $\a: X \to X'$ be a continuous onto map between compact
spaces. Assume that $(Y, \mu)$ is a uniform space, $f: X \to Y$
and $f': X' \to Y$ are maps such that $f' \circ \a=f$. Then $f$
is a fragmented map iff $f'$ is a fragmented map.

\item \label{fr-sep-1}
If
$f:X \to Y$ is fragmented, where $(X,\tau)$ is a Baire space and
$(Y,\rho)$ is a pseudometric space, then $f$ is continuous at the
points of a dense $G_{\delta}$ subset of $X$.

\end{enumerate}
\end{lem}
\begin{proof} (1): is straightforward.

(2): Let $A$ be a nonempty subset of $X$ and let $\eps \in \mu$.
Choose an open subset $O$ in $Y$ such that $\a(A) \cap O$ is
nonempty and $f(\a(A) \cap O)$ is $\eps$-small. Since $\a(A \cap
\a^{-1}(O))=\a(A) \cap O$ we get that $A \cap \a^{-1}(O)$ is
nonempty and $(f \circ \a) (A \cap \a^{-1}(O))$ is $\eps$-small in
$Y$. This completes the proof because $\a^{-1}(O)$ is open in $X$
by the continuity of $\a$.

(3):
Consider the weak uniformity $\mu_w$ on $Y$ generated by the
system $\{f_i: Y \to Z_i\}_{i \in I}$. Since this system separates
the points of $Y$ and each $Z_i$ is a Hausdorff uniform space we get that
$\mu_w$ is a Hausdorff uniformity on $Y$. Furthermore $\mu_w$ is
continuous on $Y$. Now it is clear that $\mu_w$ coincides with the
unique compatible uniformity on the compact space $Y$. The system
of entourages $$\{(f_i^{-1} \times f_i^{-1}) (\eps_i) | \ \ \eps_i \in \mu_i\}$$ is a
subbase of the uniformity $\mu_w$. Now use Remark \ref{r:fr}.2.

(4):
This is \cite[Lemma 6.4]{GM} which in turn was inspired by
Lemma 2.1 of Namioka's paper \cite{N}.


(5):
If $f'$ is fragmented then $f$ is fragmented by (2).
If $f$ is fragmented then $f'$ is fragmented by (4)
(with $Y=Y'$ and $\nu=id_Y$).

(6): For a fixed $\ep > 0$ consider
$$
O_{\ep}:= \{ \text{union of all $\tau$-open subsets $O$ of $X$
with} \hskip 0.3cm {\diam}_{\rho} f(O) \leq \ep \}.
$$
The
fragmentability implies that $O_{\ep}$ is dense in $X$. Clearly
$\bigcap \{O_{\frac{1}{n}} :   n \in \N \}$
serves as the required dense
$G_{\delta}$ subset of $X$.
\end{proof}

\subsection{Baire class one functions}

Given two topological spaces $X$ and $Y$, a function $f:X\to Y$ is
{\em of Baire class 1} or more briefly {\em Baire 1} if the
inverse image of every open set in $Y$ is $F_\sigma$ (the union of
countably many closed sets) in $X$. In general a Baire 1 function
need not be the same as
a limit of a sequence of continuous
functions.
The following results are well known. Mainly they are classical
and come from R. Baire. See for example \cite{Dulst, Kech, GMU}.

\begin{lem} \label{l:Baire1}
\ben
\item
If $Y$ is metrizable and $\{f_n : X \to Y\}_{n \in \N}$ is a
sequence of continuous functions converging pointwise to $f : X
\to Y$ then f is Baire 1.
\item
If $X$ is separable and metrizable then a real valued function $f: X \to \R$
is Baire 1 if and only if $f$ is a pointwise limit of a sequence of
continuous functions.
\een
\end{lem}

%
%

We denote by $\calB_1(X)$ the space of all real valued Baire 1
functions on $X$ equipped with the pointwise convergence topology.
That is, $\calB_1(X)$ is a topological subspace of the product
space $\R^X$.

As usual, a space is {\em Baire} if the intersection of any
countable family of dense open sets is dense.
{\it Hereditarily
Baire} means that every closed subspace is a Baire space.

\begin{lem} \label{Baire-1}
Let $(X,\tau)$ be a hereditarily Baire (e.g., Polish, or compact)
space, $(Y,\rho)$ a pseudometric space. Consider the following
assertions:
\begin{enumerate}
\item [(a)]
$X$ is $(\tau,\rho)$-fragmented by $f: X \to Y$;
\item [(b)]
$f$ is barely continuous;
\item [(c)]
$f$ is of Baire class 1.
\end{enumerate}
\ben
\item
Then $(a) \Leftrightarrow (b)$.
\item
If $X$ is Polish and $Y$ is a separable metric space then
$(a)
\Leftrightarrow (b) \Leftrightarrow~(c).$
 \een
\end{lem}
\begin{proof} For (a) $\Leftrightarrow$ (b) use
Lemma \ref{simple-fr}.

The equivalence (b) $\Leftrightarrow$ (c) for Polish $X$ and
separable metric space
$Y$ is well known (see \cite[Theorem 24.15]{Kech}) and
goes back to Baire.
\end{proof}

\begin{cor} \label{c:fr=B1}
\ben
\item ${\mathcal F}(X)=B_r'(X)$
for every compact space $X$.
\item ${\mathcal F}(X)=\calB_1(X)=B_r'(X)$
for every Polish space $X$.
\een
\end{cor}

\sk
\subsection{Fragmented families}

The following definition was introduced in \cite{GM} and
independently in the Ph.D. Thesis of M.M. Guillermo \cite{Guil}.

\begin{defn} \label{d:fr-family}
\ben
\item
We say that a {\it family of functions} $F=\{f: (X,\tau) \to
(Y,\mu) \}$ is {\it fragmented}
if the condition of Definition
\ref{def:fr}.1 holds simultaneously for all $f \in F$. That is,
$O \cap A$ is nonempty and
$f(O \cap A)$ is $\eps$-small for every $f \in F$. It is
equivalent to say that the mapping
$$
\pi_{\sharp}: X \to Y^F, \hskip 0.4cm \pi_{\sharp}(x)(f)=f(x)
$$
is $(\tau, \mu_u)$-fragmented, where $\mu_u$ is the uniform
structure of uniform convergence on the set $Y^F$ of
all mappings from $F$ into $(Y, \mu)$.
\item
Analogously one can define the notion of a {\it barely continuous
family}. The latter means that every closed nonempty subset
$A\subset X$ contains a point $a\in A$ such that $F_A =\{f
\upharpoonright_A: f\in F \}$ is equicontinuous at $a$.
If $\mu$ is pseudometrizable then so is $\mu_u$. Therefore if in
addition $(X, \tau)$ is hereditarily Baire then it follows by
Lemma \ref{Baire-1}.1
that $F$ is a fragmented family if and
only if  $F$ is a barely continuous family.
\een
\end{defn}

Fragmented families, like equicontinuous families, are stable
under pointwise closures as the first assertion of the following lemma shows.

\begin{lem} \label{l:fr-f-cls}
\ben
\item
Let $F=\{f_i: (X,\tau) \to (Y,\mu) \}_{i \in I}$ be a fragmented family of
functions.
Then the pointwise closure
$\cls_p{F}$ of $F$ in $Y^X$ is
also a $(\tau, \mu)$-fragmented family.
\item
Let $\a: X \to X'$ be a continuous onto map between compact
spaces. Assume that $(Y, \mu)$ is a uniform space, $F:=\{f_i: X \to Y\}_{i \in I}$
and $F':=\{f'_i: X' \to Y\}$ are families such that $f'_i \circ \a=f_i$ for every $i \in I$. Then $F$
is a fragmented family iff $F'$ is a fragmented family.
\een
\end{lem}
\begin{proof} (1)
A straightforward ``$3\eps$-argument".

(2)
If $F'$ is fragmented then an obvious modification (for families) of Lemma \ref{simple-fr}.2 shows that $F$ is fragmented.

If $F$ is fragmented then this means that the identity map $(X,\tau) \to (X,\xi)$ is fragmented, where
$\xi$ is the weak uniformity induced on the set $X$ by the natural map
$X \to (Y^F, \mu_u)$ (see Definition \ref{d:fr-family}). 
Analogously, the map $X' \to (Y^{F'}, \mu_u)$ induces the uniformity $\xi'$ on $X'$.
Then $\a: (X,\xi) \to (X',\xi')$ is a uniform map.
Now Lemma \ref{simple-fr}.4 implies that the identity map
$(X',\tau') \to (X',\xi')$ is fragmented.
This means in view
of
 Definition \ref{d:fr-family} that $F'$ is a fragmented family.
\end{proof}

\begin{lem} \label{l:FrFa}
\ben
\item
Suppose $F$ is a compact space, $X$ is \v{C}ech-complete,
$M$ is a metrizable space and we are given a separately continuous map
$w: X \times F \to M$.
Then the naturally associated family
$$\tilde{F}:=\{\tilde{f}: X \to M\}_{f \in F}$$ is fragmented, where $\tilde{f}(x) = w(x,f)$.
\item
Suppose $F$ is a compact and metrizable space, $X$ is Polish and $M$ is separable and metrizable.
Assume we are given a map $w: X \times F \to M$ such that
every $\tilde{x}: F \to M, f \mapsto w(x,f)$ is continuous and $y: X \to M$ is continuous at
every $\tilde{y} \in Y$ for some dense subset $Y$ of $F$. Then the family $\tilde{F}$ is barely
continuous (hence, fragmented).
\een
\end{lem}
\begin{proof}
(1): Use Namioka's joint continuity theorem \cite{N-jct}.

(2): Since every $\tilde{x}: F \to M$ is continuous, the natural map $j: X
\to C(F, M), \ j(x)=\tilde{x}$ is well defined.
For every closed nonempty
subset $A \subset X$ the induced map $j |_{A}: A \to
C(F, M)$ has a point of continuity by virtue of
\cite[Proposition 2.4]{GMU}, where $C(F,M)$ carries the
sup-metric.
Hence, $\tilde{F}_{A} =\{\tilde{f}
\upharpoonright_{A}: A \to M \}_{f \in F}$ is equicontinuous at
some point $a\in A$. This means, by Definition \ref{d:fr-family}.2, that
the family $\tilde{F}$ is barely continuous.
\end{proof}

\begin{defn} \label{d:ev-fr}  We say that a family of functions
$F=\{f: (X,\tau) \to (Y,\mu) \}$ is
{\it eventually fragmented} if every sequence in $F$ has a subsequence
which is a fragmented family on $X$.
\end{defn}


\begin{example} \label{e:fr-fam}
Let $V$ be a Banach space. Then we can treat
the unit ball $B_V$ of $V$ as a family of
functions on the weak${}^*$ compact space $B_{V^*}$. \ben
\item
$B_V$ is a fragmented family of functions on $B_{V^*}$
if and only if $V$ is Asplund. This fact easily follows from the
following well known characterization of Asplund spaces: $V$ is
Asplund iff $B_{V^*}$ is (weak${}^*$,norm)-fragmented (Remark
\ref{r:fr}.4).
\item
$B_V$ is an eventually fragmented family of functions on $B_{V^*}$
if and only if $V$ is a Rosenthal Banach space (see Proposition
\ref{p:crit}).
\een
\end{example}

\section{Dynamical background: $G$-spaces and Banach representations}
\label{s:Dyn}

By a $G$-space (or a $G$-system)
$X$ we mean a
continuous action of a topological group $G$ on a topological
space $X$. Sometimes we write it as $(G,X)$.
The Banach algebra (under the supremum norm) of all continuous
real valued  bounded functions on a topological space $X$ will be
denoted by $C(X)$.
The material in this
section is mostly well known. For more details and undefined
concepts see for example \cite{GM, G4, GM-suc}.

\subsection{Cyclic $G$-spaces}
\label{s:cycl}

A function $f\in C(X)$ on a $G$-space $X$ {\em comes from\/} a
compact $G$-space $Y$ if there exist a $G$-compactification $\nu:
X \to Y$ (so, $\nu$ is onto if $X$ is compact) and a function
$f_0 \in C(Y)$ such that $f=f_0 \circ \nu$.
Then necessarily, $f$ is right uniformly continuous
(notation: $f \in$ RUC$(X)$) that is, the orbit map
$f: G \to C(X), \ g \mapsto fg$ is norm continuous.

%
%
For every $G$-space $X$ a function $f: X \to \R$
lies in RUC$(X)$ iff it comes from a compact $G$-space $Y$.
Among all possible $G$-compactifications $\nu: X \to Y$ of a
$G$-space $X$ such that $f$ comes from $(\nu,Y)$ there exists the
smallest one. One may define it by the smallest closed unital $G$-subalgebra
${\mathcal A_f}$ of RUC$(X)$ generated by the orbit $fG$ of $f$ in
RUC$(X)$. 
Denote by $X_f$ the Gelfand space
of the algebra ${\mathcal A_f}$. Then the corresponding
$G$-compactification ${\a}_f: X \to Y:=X_f$ is the
required
one. We call
$X_f$ the
\emph{cyclic $G$-system} of $f$. 
The function $f$ comes from $X_f$.
There exists a continuous function $f_0: X_f \to \R$
such that $f=f_0 \circ {\a}_f$ and furthermore
the family
$f_0G$ separates points of $X_f$.

\subsection{Enveloping semigroups}
\label{s:env}

The {\em enveloping (or Ellis) semigroup\/} $E=E(G,X)=E(X)$ of a
compact
dynamical $G$-system $X$ is defined as the closure in $X^X$ (with
its compact pointwise convergence topology) of the set
$\breve{G}=\{\breve{g}: X \to X\}_{g \in G}$ of translations.
With the operation of composition
of maps $E(X)$ is a right topological semigroup. Moreover, the map
$$j=j_X: G \to E(X),\ g \mapsto \breve{g}$$ is a right topological
semigroup compactification of $G$. The compact space $E(X)$
becomes a $G$-space with respect to the natural action
$$G \times E(X) \to E(X), \ \ \ \ (gp)(x)=gp(x).$$

Let $E=E(X)$ be the enveloping semigroup of a compact $G$-system
$X$. For every $f \in C(X)$ define
$$
E^f:=\{f \circ p: X \to \R \}_{p \in E}.
$$

Then $E^f$ is a pointwise compact subset of $\R^X$, being a
continuous image of $E$ under the map $q_f: E \to E^f, \hskip
0.2cm p \mapsto f \circ p$. The orbit $fG$ of $f$ is a dense subset of
$E^f$. It follows that $E^f=\cls_p(fG)$.

%
%

\subsection{Banach representations of dynamical systems}
\label{s:Ban Rep}

Let $V$ be a Banach space. Denote by $\Iso(V)$ the topological
group of all linear isometries of $V$ onto itself, equipped with the
pointwise convergence topology.

\begin{defn} \label{d:repr} \cite{Menz} Let $X$ be a $G$-space.
A continuous \emph{representation} of $(G,X)$ on a Banach space
$V$ is a pair
$$
h: G \to {\Iso}(V), \ \ \ \alpha: X \to V^*
$$
where
$h: G \to {\Iso}(V)$ is a continuous co-homomorphism and
$\alpha: X \to V^*$ is a weak$^*$ continuous bounded
$G$-map with respect to the {\em dual action\/} $G \times V^* \to V^*,
\ (g\varphi)(v):=\varphi(h(g)(v))$.
We say that a continuous representation $(h,\alpha)$ is
\emph{faithful} when $\alpha$ is a topological embedding.
\end{defn}

Every compact $G$-space $X$ admits a canonical faithful representation
on the Banach space $C(X)$. A natural question is to characterize
dynamical systems according to their representability on nice
Banach spaces.

\subsubsection{Reflexive representations and WAP systems}

A compact dynamical system $(G,X)$ is weakly almost periodic (WAP)
if $C(X)=WAP(X)$. As usual a continuous function $f: X \to \R$ is
WAP if the weakly closure of the orbit $fG$ is weakly compact in
$C(X)$.
 A compact $G$-space $X$ is WAP iff every element $p \in E(X)$ is
 a continuous selfmap of $X$ (see Ellis and Nerurkar \cite{EN}).


\begin{fact}
\emph{(See \cite{Menz})}
Let $X$ be a compact metric $G$-space. The following conditions
are equivalent:
\begin{enumerate}
\item
$(G,X)$ is weakly almost periodic (WAP).
\item
$(G,X)$ is reflexively representable (that
is, admits a faithful representation on a
reflexive Banach space).
\end{enumerate}
\end{fact}

\subsubsection{Asplund representations, RN and HNS systems}
\label{subs:RN}

A dynamical system is {\em Radon--Nikod\'ym} (RN) if it admits a
faithful
representation on an Asplund Banach space \cite{Menz, GM}.
If $G=\{1\}$, we get
the class of Radon--Nikod\'ym compact spaces in the sense of
Namioka \cite{N}.

We recall the concept of non-sensitivity
(see for instance \cite{GW1,AAB,Menz,GM,GM-fixed} and the references
therein). Let $d$ be a compatible metric on a compact $G$-system
$X$. We say that $(G,X)$ is {\em non-sensitive} if for every
$\ep>0$ there exists a non-empty open set $O\sbs X$ such that for
every $g\in G$ the set $gO$ has $d$-diameter $<\ep$.
%
$(G,X)$ is {\em hereditarily non-sensitive} (HNS)
if all its closed $G$-subspaces are non-sensitive.

For a nonmetrizable version of HNS in terms of uniform structures
and some related properties we refer to \cite{GM}.

\begin{fact} \label{t:AspRep} \emph{(See \cite{GM} and \cite{GMU})}
 Let $X$ be a compact metric $G$-space. The
following conditions are equivalent:
\begin{enumerate}
\item
$(G,X)$ is hereditarily non-sensitive (HNS).
\item
$(G,X)$ is RN, that is, admits a
faithful
representation on an Asplund Banach space.
\item
the enveloping semigroup $E(X)$ is metrizable.
\end{enumerate}
\end{fact}

\begin{fact} \cite{GM} \label{t:HNS} Let $X$ be a compact $G$-space. The
following conditions are equivalent:
\begin{enumerate}
\item
$(G,X)$ is HNS.
\item
$(G,X)$ is RN-approximable, that is, admits
sufficiently many representations on Asplund Banach spaces.
\item $E(X)=\{p: X \to X\}_{p \in E(X)}$ is a fragmented family.
\item $\breve{G}=\{\breve{g}: X \to X\}_{g \in G}$
is a fragmented family.
\end{enumerate}
\end{fact}

\section{Banach space background and the dynamical BFT dichotomy}
\label{s:Ban}

\subsection{Bourgain-Fremlin-Talagrand theorems}

Recall that a topological space $K$ is a {\it Rosenthal compactum}
\cite{Godefroy} if it is homeomorphic to a pointwise compact
subset of the space $\calB_1(X)$ of functions of the first Baire
class on a Polish space $X$. All metric compact spaces are
Rosenthal. An example of a separable non-metrizable Rosenthal
compact is the {\it Helly compact}
\cite{Eng}
of all (not only strictly)
increasing selfmaps of $[0,1]$ in the pointwise topology. Another
is the ``two arrows" space
\cite{Eng}
of Alexandroff and Urysohn.
Recall that a topological space $K$ is {\em Fr\'echet}
(or, {\em Fr\'echet-Urysohn} \cite{Eng})
 if for every $A\subset K$ and every $x\in \cls(A)$ there exists a
sequence of elements of $A$ which converges to $x$.

The following theorem is due to Bourgain-Fremlin-Talagrand
\cite[Theorem 3F]{BFT}, generalizing a result of Rosenthal. The
second assertion (BFT dichotomy) is presented as in the book of
Todor\u{c}evi\'{c} \cite{To-b} (see Proposition 1 of Section 13).

\begin{fact} \label{BFT}
\begin{enumerate}
\item
Every Rosenthal compact space $K$ is Fr\'echet.
\item (BFT dichotomy)
Let $X$ be a Polish space and let $\{f_n\}_{n=1}^\infty \subset
C(X)$ be a
pointwise bounded sequence of bounded functions.
Let $K$ be the
pointwise closure of $\{f_n\}_{n=1}^\infty$ in $\R^X$. Then either
$K \subset \calB_1(X)$ (i.e. $K$ is Rosenthal compact) or $K$
contains a homeomorphic copy of $\beta\N$.
\end{enumerate}
\end{fact}

Clearly, $\beta \N$ the Stone-\v{C}ech compactification of the
natural numbers $\N$, is not Fr\'echet, and hence it
is not a Rosenthal compact space.


\begin{defn} \label{d:Ros-F}
Let $X$ be a topological space. We say that a subset $F\subset
C(X)$ is a \emph{Rosenthal family} (for $X$) if $F$ is norm bounded and
the pointwise closure $\cls_p(F)$ of $F$ in $\R^X$ consists of fragmented maps,
that is,
$$\cls_p(F) \subset {\mathcal F}(X).$$
\end{defn}

In the following result we combine two theorems from Talagrand's book \cite{Tal}.
Here we reformulate assertion (3)
in terms of ${\mathcal F}(X)$ using the equality ${\mathcal F}(X)=B_r'(X)$
(Corollary \ref{c:fr=B1}.1).

\begin{fact} \label{t:Talagr}
\emph{(Talagrand \cite{Tal})}
Let $X$ be a compact space
and $F \subset C(X)$ a bounded subset.
The following conditions are equivalent:
\begin{enumerate}
\item
$F$ does not contain a subsequence equivalent to the unit
basis of $l_1$.
\item
Each sequence in $F$ has a pointwise convergent subsequence in
$\R^X$ (i.e., $F$ is sequentially precompact in $\R^X$).
\item
${\cls}_p(F) \subset {\mathcal F}(X)$
(i.e., $F$ is a Rosenthal family for $X$).
\item
Every sequence in $F$ has a weak-Cauchy subsequence.
\item Every countable subfamily $S$ of $F$ is a Rosenthal family for $X$.
%
\end{enumerate}
\end{fact}
\begin{proof}
The equivalence of (1), (2) and (3)
is a part of \cite[Theorem 14.1.7]{Tal}.
The equivalence of (2) and (4) is a part of \cite[Theorem 7.3.1]{Tal}.
The equivalence of the assertions (1), (2), (3) and (4) implies now
that each of them is equivalent also to (5).
\end{proof}

Let $F\subset C(X)$ be a norm bounded subset. Then the pointwise
closure $\cls_p(F)$ in $\R^X$ is compact.
The following lemma
examines four natural conditions expressing ``smallness" of
$F$ (compare Proposition \ref{p:3f-on}).

\begin{lem} \label{l:4}
Let $F\subset C(X)$ be a norm bounded family where $X$ is a Polish space
and $\cls_p(F)$ is the (necessarily compact) pointwise closure of $F$ in $\R^X$.
Consider the following conditions:
\bit
\item [(a)]
$\cls_p(F) \subset C(X)$.
\item [(b)]
$\cls_p(F)$ is a 
metrizable subspace in $\R^X$.
\item [(c)]
$F$ is a fragmented (equivalently, barely continuous) family of
functions on $X$.
\item [(d)]
$\cls_p(F) \subset \calB_1(X)$. \eit Then always  $(a)
\Longrightarrow (b) \Longleftrightarrow (c) \Longrightarrow (d)$.

\end{lem}
\begin{proof} (a) $\Rightarrow$ (b):
Let $Y$ be a dense countable subset of $X$. Since every function
$\phi \in \cls_p(F)$ is continuous we get that the natural
continuous projection $\R^X \to \R^Y$ induces an injection on
$\cls_p(F)$. Since $\cls_p(F)$ is compact we get its homeomorphic
embedding into the second countable space $\R^Y$.

(b) $\Rightarrow$ (c):
Apply Lemma \ref{l:FrFa}.2 to the evaluation map $X \times \cls_p(F) \to \R$.
It follows that $\cls_p(F)$ and hence $F$ are fragmented families.


(c) $\Rightarrow$ (b): 
 The family $F$ is fragmented means by
Definition \ref{d:fr-family} that the natural map $X \to \R^F$ is
fragmented, where $ \R^F$ carries the uniformity of uniform
convergence. Then the image of $X$ is separable as it follows by
\cite[Lemma 6.5]{GM}. Now \cite[Theorem 4.1]{N} implies that the
pointwise closure $\cls_p(F)$ of $F$ in $\R^X$ is metrizable.

(c) $\Rightarrow$ (d): Since $F$ is a fragmented family its
pointwise closure $\cls_p(F)$ is again a fragmented family
(Lemma \ref{l:fr-f-cls}.1). In particular, every member $\phi \in
\cls_p(F)$ is a fragmented map on $X$. Since $X$ is Polish this
means
by Corollary \ref{c:fr=B1} that $\phi \in \calB_1(X)$.
\end{proof}


\begin{lem} \label{l:genQuot}
Let $q: X_1 \to X_2$ be a map between
topological spaces. Then
\ben
\item
The natural map
$\gamma: \R^{X_2} \to \R^{X_1}, \  \gamma(\phi)=\phi \circ q$
is pointwise continuous.
\item If $q: X_1 \to X_2$ is onto then $\gamma$ is injective.
\item
Let $q: X_1 \to X_2$ be a continuous onto map between compact spaces,
$F_2 \subset C(X_2)$ and $F_1 \subset C(X_1)$ be norm
bounded subsets such that $F_1=F_2 \circ q$. Then
\begin{itemize}
\item [(a)]
$F_1$ is a Rosenthal family for $X_1$ if and only if $F_2$ is a
Rosenthal family for $X_2$.
\item [(b)]
$\gamma$ induces a
homeomorphism between the compact spaces $\cls_p(F_2)$ and
$\cls_p(F_1)$.
\end{itemize}
\een
\end{lem}
\begin{proof}  Claims (1) and (2) are trivial.

(3)(a): By the continuity of $\gamma$ we get $\gamma(F_2) \subset
\gamma(\cls_p(F_2)) \subset \cls_p(\gamma(F_2))$. Since $F_2$ is
bounded the set $\cls_p(F_2)$ is compact in $\R^{X_2}$. Then
$\gamma(\cls_p(F_2))=\cls_p(\gamma(F_2))$. On the other hand,
$\gamma(F_2) = F_2 \circ q = F_1$. Therefore, $\cls_p(F_2) \circ q
= \cls_p(F_1)$. Now apply Lemma \ref{simple-fr}.5.

(3)(b): Combine the assertions (1) and (2) taking into account
that $\gamma(\cls_p(F_2))=\cls_p(\gamma(F_2))=\cls_p(F_1)$.
\end{proof}


\begin{prop} \label{p:aft} Let $X$ be a compact space and
$F$ a bounded
subset of $C(X)$.
The following conditions are equivalent:
\begin{enumerate}
\item
$F$ is a Rosenthal family for $X$.
\item
Every sequence in $F$ has a subsequence which is a fragmented
family on $X$ (i.e. $F$ is an eventually fragmented family of maps on
$X$).
\end{enumerate}

\sk

If $X$ is metrizable then each of these conditions is
equivalent to the following:

\sk

\emph{(3)}  The pointwise closure $\cls_p(F)$ of $F$ in $\R^X$ is
a subset of $\calB_1(X)$.

\end{prop}
\begin{proof}
(1) $\Rightarrow$ (2): Let $S$ be a sequence in $F$. By the
implication (3) $\Rightarrow$ (2) from Fact \ref{t:Talagr} we can
choose a pointwise convergent subsequence $\{f_n\}_{n \in \N}$ of $S$
with $\phi=\lim f_n$. Denote by $K$ the compact metrizable
subset $\{f_n\}_{n \in \N} \cup \{\phi\}$ in $\R^X$.
Define the pointwise continuous map
$$q: X \to \R^S, \ q(x)(f)=f(x) \ \ \ \forall f \in S$$
and denote by $X'$ the subspace $q(X) \subset \R^S$.
Clearly, $X'$ is pointwise compact because $S$ is norm bounded.
Furthermore, $X'$ is metrizable since $S$ is countable.
 For every $f \in S$ we have the uniquely defined
continuous map $f': X' \to \R$, $f'(q(x))=f(x)$ such that $f=f' \circ q=\gamma(f')$.
By Lemma \ref{l:genQuot}.3(b), $\gamma: \R^{X} \to \R^{X'}$ induces a
homeomorphism between the compact spaces
$\cls_p(A') \to \cls_p(A)$, where $A:=\{f_n\}_{n \in \N}$ and $A':=\{f'_n\}_{n \in \N}$. Therefore,
there exists $\psi \in X'$ with $\gamma(\psi)=\phi$ such that $\gamma$ induces
a homeomorphism $K' \to K$, where $K':=\{f'_n\}_{n \in \N} \cup \{\psi\}$.
Consider the evaluation map $X' \times K' \to \R$. Then we can apply Lemma
\ref{l:FrFa}.2 which implies that $K'$, and hence also its
subfamily $A'$, are fragmented families of maps on $X'$.
Now Lemma \ref{l:fr-f-cls} .2 implies that $A$ is a fragmented family of maps on $X$.

(2) $\Rightarrow$ (1): We have to show that $F$ is a Rosenthal family for $X$.
By Fact \ref{t:Talagr}
it is equivalent to check that every sequence $S$ in $F$ has a pointwise convergent subsequence in
$\R^X$.
By our assumption there exists a subsequence of $S$ which
is a fragmented family of functions on $X$. So without restriction of generality
we may assume that $S$ itself is a fragmented family.
 As in the proof above consider the quotient $q: X \to X' \subset \R^S$.
Then the family $S':=\{f': X' \to \R\}_{f \in S}$ is a fragmented family by Lemma \ref{l:fr-f-cls}.2.
Now by Lemma \ref{l:4} the
pointwise closure $\cls_p(S')$ is a (compact)
metrizable subspace in $\R^{X'}$. Therefore there exists a
convergent subsequence of $S'$
(in $\cls_p(S') \subset \R^{X'}$).
By Lemma \ref{l:genQuot}.3(b), $\gamma: \R^{X} \to \R^{X'}$ induces a
homeomorphism between the compact spaces $\cls_p(S)$ and
$\cls_p(S')$. Hence there exists a convergent subsequence of $S$ (in $\R^X$), as desired.

\sk (1) $\Leftrightarrow$ (3) (For metrizable $X$): Since $X$ is
compact metrizable we have ${\mathcal F}(X)=\calB_1(X)$ by Corollary
\ref{c:fr=B1}.2.
\end{proof}

Let $X=\{0,1\}^{\N}$ be the Cantor cube and $F:=\{p_n\}_{n \in\N}$
the family of all projection mappings (with $p_n(x)=x(n)$). It is
well known that the pointwise closure of $F$ in $\R^X$ is
homeomorphic to $\beta \N$ (see for instance \cite[p. 4]{To-b}).
By Fact \ref{BFT} it follows that $F$ is not a Rosenthal family
for $X=\{0,1\}^{\N}$.

\sk
\subsection{Banach spaces not containing $l_1$}
\label{subsec-l1}

\begin{defn}\label{Ros-Banach}
Let us say that a Banach space $V$ is {\em Rosenthal} if it does
not contain an isomorphic copy of $l_1$.
\end{defn}

Clearly the class of Asplund spaces (see Remarks \ref{r:fr}.4) is
a subclass of the class of Rosenthal spaces.  The difference
between these two classes can be illustrated in terms of
fragmentability. Compare the last two items of Remarks \ref{r:fr}
and Proposition \ref{p:crit}.

Recall the following famous result of Rosenthal.

\begin{fact} \label{t:Ros0} \emph{(Rosenthal \cite{Ros0})}
Let $V$ be a Banach space. The following conditions
are equivalent:
\begin{enumerate}
\item
$V$ is a Rosenthal Banach space.
\item
Every bounded sequence in $V$ has a weak-Cauchy subsequence.
\end{enumerate}
\end{fact}


Every Banach space $V$ can be treated as a natural subspace of the Banach space
$C(B_{V^*})$, where $B_{V^*}$ is the weak$^*$
compact unit ball of $V^*$. Furthermore,
the corresponding weak topology on $V$
coincides with the weak topology inherited from $C(B_{V^*})$.
Therefore
taking into account Fact \ref{t:Talagr} and Definition \ref{d:Ros-F}
we get the following reformulation of Fact \ref{t:Ros0}.

\begin{lem} \label{l:crit}
Let $V$ be a Banach space. The following conditions
are equivalent:
\begin{enumerate}
\item
$V$ is a Rosenthal Banach space.
\item
The unit ball $B_V$ of $V$ is a Rosenthal family for the weak$^*$
compact unit ball $B_{V^*}$ of $V^*$.
\end{enumerate}
\end{lem}

For the separable case, we have the following theorem.

\begin{fact} \label{t:Ros}
\emph{(\cite{OR}, \cite[p. 374]{Ros1}  and \cite[Theorem
3]{Ros2})} Let $V$ be a separable Banach space.
The following conditions are equivalent:
\begin{enumerate}
\item $V$ is a Rosenthal Banach space.
\item
${\card}(V^{**})={\card}(V)$.
\item
Every $v^{**} \in V^{**}$ is a Baire 1 function when
restricted to $B_{V^*}$.
\end{enumerate}
\end{fact}

Thus a separable Banach space $V$ does not contain an isomorphic
copy of $l_1$ if and only if every element $x^{**} \in V^{**}$ is
Baire 1 when restricted to the unit ball $B_{V^*}$ with its
weak$^*$ topology $\sigma(V^*,V)$. This classical result of Odell
and Rosenthal was generalized in \cite{SS}.

Let $A$ be a weak$^*$ compact subset of a dual Banach space
$V^*$. Following \cite{RSU} we say that $A$ has the \emph{scalar
point of continuity property} if for each weak$^*$ compact subset
$M$ of $A$ and every $x^{**} \in V^{**}$, the restriction
$x^{**}|_M$ of $x^{**}$ to $M$ has a point of continuity.

\begin{fact} \label{t:SS} \emph{(E. Saab and P. Saab \cite{SS})}
A Banach space $V$ (separable or not) is Rosenthal if and only if
$B_{V^*}$ has the scalar point of continuity property.
\end{fact}

The following result gives three characterizations of Rosenthal
spaces in terms of fragmentability.

\begin{prop} \label{p:crit}
Let $V$ be a Banach space. The following conditions are
equivalent:
\begin{enumerate}
\item
$V$ is a Rosenthal Banach space.
\item
Each $x^{**} \in V^{**}$ is a fragmented map when restricted to the
weak${}^*$ compact ball $B_{V^*}$.
\item
Every bounded
subset $A$ of the dual $V^*$ is $(\tau_{w^*}, \mu_w)$-fragmented
(that is, $id_A: (A,\tau_{w^*}) \to (A,\mu_w)$ is
fragmented)
where $\tau_{w^*}$ is the weak${}^*$ topology and $\mu_w$ is the
weak uniformity on $A$.
\item
$B_V$ is an eventually fragmented family of functions on $(B_{V^*}, w^*)$.
\end{enumerate}
\end{prop}
\begin{proof}
By Corollary \ref{c:fr=B1}.1, ${\mathcal F}(X)=B_r'(X)$ for
$X=B_{V^*}$. Now Fact \ref{t:SS} yields the equivalence (1)
$\Leftrightarrow$ (2).

(1) $\Leftrightarrow$ (3): This follows by Fact \ref{t:SS} and
Lemma \ref{Baire-1}.1 taking into account Remark \ref{r:fr}.2.

For (1) $\Leftrightarrow$ (4) use
Proposition \ref{p:aft}
and Lemma \ref{l:crit}.
\end{proof}

\begin{remark} \label{r:WRN}
\ben
\item
The equivalence of (1) and (2) in Proposition \ref{p:crit} is
indeed a natural generalization of the Odell-Rosenthal result \cite{OR}
because for compact metrizable $X$ we have $\calB_1(X)={\mathcal
F}(X)$ (Corollary \ref{c:fr=B1}.2) and the weak${}^*$ compact ball
$B_{V^*}$ is metrizable for separable~$V$.
\item
Let $V$ be a Banach space and $A$ a weak${}^*$ compact
absolutely convex subset of $V^*$. Then by \cite[Theorem 9]{RSU},
$A$ has the scalar point of continuity property if and only if $A$
is a \emph{weak Radon-Nikodym subset} (WRN for short). We
refer to \cite{Bo, SS, RSU} for exact definitions and additional
information about WRN subsets.
See also Theorem \ref{t:WRN} below about WRN compact spaces.
\een
\end{remark}


\subsection{Convex hulls}

The following result is proved in
Bourgain-Fremlin-Talagrand \cite{BFT}.

\begin{fact} \label{t:BFT2} \cite[Theorem 5E]{BFT}
Let $X$ be a complete metric space, $A \subset \calB_1(X)$ a pointwise
compact uniformly bounded set.
Then its convex hull ${\co}(A)$ is relatively compact in $\calB_1(X)$
(equivalently, $\cls_p({\co}(A)) \subset \calB_1(X)$).
\end{fact}

For Rosenthal families we get the following result.

\begin{prop} \label{c:conv}
Let $F$ be a Rosenthal family for a compact space $X$. Then its
convex hull ${co}(F)$ is also a Rosenthal family for $X$.
\end{prop}
\begin{proof} \emph{First case:}
For a compact metrizable $X$
combine Fact \ref{t:BFT2} and Proposition \ref{p:aft}.

\emph{Second case:} For a general compact space $X$,
by Fact \ref{t:Talagr} we have only to examine sequences. That is,
it is enough to show that every countable subset $M$ of ${co}(F)$ is a Rosenthal family.
There exists a countable subset $S \subset F$ such that $M \subset {co}(S)$.

As in the proof of Proposition \ref{p:aft} consider the quotient
$q: X \to X' \subset \R^S$ induced by the collection $S$.
Then every $f \in S$ induces a continuous map $f': X' \to \R$
such that $f=f' \circ q=\g(f')$.

By our assumption $S$ is a Rosenthal family
for $X$. Then $S':=\{f': \ f \in S\}$ is a Rosenthal family for
$X'$ (use Lemma \ref{l:genQuot}.3). Since $X'$ is metrizable we
can apply the first case and deduce that the convex hull $co
(S')$ is a Rosenthal family for $X'$. The map
$$\gamma: \R^{X'} \to \R^{X}, \ \ \gamma(\phi)=\phi \circ q$$
(see Lemma \ref{l:genQuot}) is linear and $\g(S')=S$. Therefore, $\gamma(co (S'))=co (S).$ It follows,
by Lemma \ref{l:genQuot}.3, that the collection $co (S)$, and hence its subcollection $M$, is a Rosenthal
family for $X$.
\end{proof}

\subsection{The natural affine extension map
$T: b \calB_1(X)  \to b \calB_1(B^*)$}

For every compact metric space $X$ denote by $b \calB_1(X)$ the collection
of bounded Baire 1 real valued functions on $X$. That is,
$$
b \calB_1(X)= \calB_1(X) \cap l_{\infty}(X).
$$

Then $b \calB_1(X)$ is a topological subspace of $\calB_1(X)$ with
respect to the pointwise topology (inherited from $\R^X$). One can
define a natural injective map
$$
T: b \calB_1(X)  \to b\calB_1(B^*),
$$
where $B^*$, as before, is the weak$^*$ compact unit ball of
$C(X)^*$. We will use Riesz representation theorem and Lebesgue's
Dominated Convergence Theorem.

Each $f \in b \calB_1(X)$ is universally measurable for every compact
metric space $X$ (see for
example \cite[Proposition 1F]{BFT}).
That is, for every measure $\mu \in B^*$ we can define
$$
(Tf)(\mu):= \int f d \mu.
$$

This map is well defined. Indeed, first note that when $f \in
C(X)$, $T(f)=i(f)$, where
$$
i: C(X) \hookrightarrow C(B^*), \ \ i(f)(\mu):=\langle f, \mu
\rangle =\int f d \mu
$$
is the canonical isometric inclusion of the corresponding Banach
spaces and
$$
\langle\, \cdot, \cdot \rangle: C(X) \times C(X)^* \to \R
$$
is the canonical bilinear mapping. Now if $f \in b\calB_1(X)$ then
$f$ is a pointwise limit of a sequence of continuous functions
$h_n \in C(X)$ (Lemma \ref{l:Baire1}.2). Since $f: X \to \R$ is
a bounded function we can assume in addition that the sequence
$h_n$ is uniformly bounded.
By Lebesgue's Convergence Theorem it follows that
$T(f)$ is a pointwise limit of the sequence $T(h_n)=i(h_n),$ ${n
\in \N}$. Since every $i(h_n) \in C(B^*)$ we conclude by Lemma
\ref{l:Baire1}.(2) that $T(f) \in \calB_1(B^*)$. The sequence
$i(h_n)$ is uniformly bounded in $C(B^*)$ hence $T(f)$ is a
bounded function. This means that $T(f) \in b \calB_1(B^*)$.

The map $T$ is injective because $T(f)(\delta_x)=f(x)$ for every
point mass $\delta_x \in B^*$ ($x \in X$).

\begin{remark} \label{r:functionals} 
Each $T(f)$ for $f \in b \calB_1 (X)$ can be treated as an
element of the second dual $C(X)^{**}$ of $C(X)$. Moreover the
pointwise topology of $\calB_1(B^*)$ and the weak$^*$-topology on
$C(X)^{**}$ agree on
$T(b \calB_1(X))$.
\end{remark}

\begin{lem} \label{l:seq-cont} Let $X$ be a compact metric space.
For every uniformly bounded subset $A \subset b\calB_1(X)$ the
restriction $T|_A$ of the natural injective map
$$T: b\calB_1(X)
\to b \calB_1(B^*) \cap C(X)^{**}$$ on $A$ is sequentially
continuous.
Furthermore, $T(A)$ is also uniformly bounded.
\end{lem}
\begin{proof}
Lebesgue's Convergence Theorem implies that $T$ is
sequentially continuous. The boundedness of $T(A)$ is easy.
\end{proof}

\begin{prop} \label{p:B-star}
If $F \subset C(X)$ is a Rosenthal family for a compact metric
space $X$ then
the restriction of $T$ on $\cls_p(F)$ induces a homeomorphism
$$\cls_p(F) \to \cls_p(T(F)) \subset b \calB_1(B^*) \cap
C(X)^{**}.$$
\end{prop}
\begin{proof}
As $F$ is a Rosenthal family for $X$ its pointwise closure
${\cls}_p(F)$ is a compact subset of $\calB_1(X)$. Moreover,
${\cls}_p(F)$ is a uniformly bounded subset of $b \calB_1(X)$
because $F$ is bounded (Definition \ref{d:Ros-F}). In view of
Lemma \ref{l:seq-cont} the restricted map $T: {\cls}_p(F) \to
b\calB_1(B^*)$ is sequentially continuous.
By the Bourgain-Fremlin-Talagrand theorem, Fact \ref{BFT}.1,
we know that ${\cls}_p(F)$
is Fr\'echet.
For a
Fr\'echet space a sequentially continuous map is continuous and we
conclude that the map $T: {\cls}_p(F) \to b\calB_1(B^*)$ is a
continuous injection, and therefore a homeomorphism, of
${\cls}_p(F)$ onto its image in $b\calB_1(B^*)$.
\end{proof}


\begin{prop} \label{p:from X to B*}
Let $X$ be a compact space and $F \subset C(X)$.
The following conditions are equivalent:
\ben
\item
$F$ is a Rosenthal family for $X$.
\item
$F$ is a Rosenthal family for $B^*$.
\een
\end{prop}
\begin{proof}
We use Fact \ref{t:Talagr} which depends on sequences only.
Since $F$ is a Rosenthal family for $X$,
by Fact \ref{t:Talagr} the set $F$ is sequentially
precompact
in $\R^X$ (that is, every sequence in $F$ has a subsequence which
converges in $\R^X$). Since $F$ is bounded we can apply
Lebesgue's Dominated Convergence Theorem which
implies that $F$ is sequentially
precompact
also in $\R^{B^*}$.
Thus $\cls_p(F) \subset {\mathcal F}(B^*)$ by Fact \ref{t:Talagr}.
Hence $F$ is a Rosenthal family for $B^*$.
\end{proof}

\section{Tame dynamical systems}
 \label{ss:tame}



A compact \emph{metric} dynamical $G$-system $X$ is called {\em tame}
\cite{Gl-tame}
if in the dynamical BFT-dichotomy (Fact \ref{f:D-BFT})
the first alternative occurs, i.e. $E(X)$ is Rosenthal compact
(see also Proposition \ref{p:EisTame} below).

\begin{fact} \label{GMU} \cite{GMU}
A compact metric dynamical $G$-system $X$ is tame if and only if
every element of $E(X)$ is a Baire 1 function (equivalently,
fragmented) from $X$ to itself.
\end{fact}

This result suggests the following general definition.

\begin{defn} \label{d:TAME2}
 Let $X$ be a (not necessarily metrizable) compact $G$-space.
 We say that $X$ is
\emph{tame} if for every element $p \in E(X)$ the function $p: X
\to X$ is fragmented (that is if $E(X) \subset {\mathcal F}(X,X)$).
\end{defn}

We will see later that
this class is the same as the class
of all regular systems in the sense of K\"{o}hler \cite{Ko}. In
particular this gives an enveloping semigroup characterization of
regular systems.

\begin{lem} \label{l:HNS => Tame}
Every compact HNS dynamical $G$-system is tame.
\end{lem}
\begin{proof}
$E(X)=\{p: X \to X\}_{p \in E(X)}$ is a fragmented family when $X$
is HNS by Fact \ref{t:HNS}. In particular we get $E(X) \subset
{\mathcal F}(X,X)$.
\end{proof}

Roughly speaking the difference between HNS and tame systems is
the difference between ``fragmented families" and ``families which
consist of fragmented maps" (see Facts \ref{t:HNS} and
\ref{GMU}).

\begin{lem} \label{l:tame-prop} For every $G$
the class of tame $G$-systems is closed under subsystems,
arbitrary products and factors.
\end{lem}
\begin{proof}
The case of subsystems is trivial because the fragmentability of
maps is a hereditary property. The cases of products and factors
both can be proved using Lemma \ref{simple-fr}.

\emph{For factors:} let $\a: X \to Y$ be a $G$-factor.
By \cite[Prop. 3.8]{Ellis}
there exists a (unique) continuous onto semigroup homomorphism
$Q: E(X) \to E(Y)$ such that $\a \circ p=Q(p) \circ \a$ for every $p \in E(X)$.
Since $Q$ is onto for every $p_Y \in E(Y)$ there exists
$p_X \in E(X)$ such that the following diagram commutes.

\begin{equation*}
\xymatrix {
X \ar[d]_{\a} \ar[r]^{p_X} & X \ar[d]^{\a} \\
Y \ar[r]^{p_Y} & Y }
\end{equation*}

Then $p_X \in {\mathcal F}(X,X)$ because $(G,X)$ is tame. By Lemma
\ref{simple-fr}.4 we obtain that $p_Y \in {\mathcal F}(Y,Y)$. This
shows that $(G,Y)$ is tame.

\sk

\emph{For products:} let $X:=\prod_{i \in I} X_i$ be a $G$-product of compact
tame $G$-spaces $X_i$ with canonical $G$-projections $\a_i: X \to
X_i$. For every $p \in E(X)$ and every index $i$ we have the
following commutative diagram

\begin{equation*}
\xymatrix {
X \ar[d]_{\a_i} \ar[r]^{p} & X \ar[d]^{\a_i} \\
X_i \ar[r]^{p_i} & X_i }
\end{equation*}
for some $p_i \in E(X_i)$.
By our assumption $p_i \in {\mathcal
F}(X_i,X_i)$ because $(G,X_i)$ is tame. Hence each $p_i \circ \a_i$
belongs to ${\mathcal F}(X,X)$ (Lemma \ref{simple-fr}.2). Then the
same is true for $\a_i \circ p$. The family of projections $\{\a_i\}_{i \in I}$
separates points of $X$. Now directly from Lemma \ref{simple-fr}.3
we conclude that $p \in {\mathcal F}(X,X)$.
\end{proof}

If $X$ is a tame $G$-space then $E(X)$ is also tame as a $G$-space.
For every $G$-space $X$ there exists a maximal tame
$G$-compactification (universal tame $G$-factor if $X$ is compact).

\begin{defn} \label{d:tameF} We say that
a continuous function $f: X \to \R$ on a
$G$-space $X$ is \emph{tame}
(notation $f \in Tame(X)$) if it comes from a tame $G$-system.
\end{defn}

Since the class of tame systems is closed under products and subsystems the collection $Tame(X)$ is a
$G$-subalgebra of $RUC(X)$ for every $G$-space $X$ (use, for example, the general approach as in \cite[Prop. 2.9]{GM}).

\begin{prop} \label{p:tame-f}
Let $X$ be a compact $G$-space, $f \in C(X)$ and
$E^f={\cls}_p(fG)$ is the pointwise closure of the orbit of $f$ in
$\R^X$.
The following conditions are equivalent:
\begin{enumerate}
\item
The function $f: X \to \R$ is tame.
\item
The cyclic $G$-space $X_f$ is tame.
\item
${\cls}_p(fG) \subset {\mathcal F}(X)$ (i.e. the orbit $fG$ is a
Rosenthal family for $X$).
\item For every countable subset $S \subset G$, $fS$ is a Rosenthal family for $X$.
\end{enumerate}
\end{prop}
\begin{proof}
The implication (2) $\Rightarrow (1)$ is obvious
because $f$ comes from the cyclic $G$-space $X_f$.

(1) $\Rightarrow (3)$: There exist: a tame compact $G$-system
$X_0$, a $G$-quotient map $q: X \to X_0$ and a function $f_0 \in
C(X_0)$ such that $f=f_0 \circ q$. By Lemma \ref{l:genQuot}.3 it
suffices to show that $f_0G$ is a Rosenthal family for $X_0$.
Clearly $f_0G$ is norm bounded in $C(X_0)$. We have to show that
the corresponding pointwise closure of $f_0G =\{f_0 \circ g: g \in
G\} \subset \R^{X_0}$ is a subset of ${\mathcal F}(X_0)$. Observe
that
$\cls_p(f_0 G)=E^{f_0}:=\{f_0 \circ p\}_{p \in E(X_0)}$.
Our $G$-system $X_0$ is tame means that every $p: X_0 \to X_0$
($p \in E(X_0)$)
is fragmented. Thus every $f_0 \circ p$ is also fragmented
(because $f_0$ is uniformly continuous). So, $f_0 G$ is a
Rosenthal family for $X_0$.

(3) $\Rightarrow (2)$: If $fG$ is a Rosenthal family for $X$ then
$E^f=\cls_p(fG) \subset {\mathcal F}(X)$. This means that $f \circ p \in {\mathcal
F}(X)$ for every $p \in E$. Consider the cyclic $G$-system $X_f$
and the natural $G$-quotient ${\a}_f: X \to X_f$.
By elementary properties of cyclic $G$-spaces (Section \ref{s:Dyn})
there exists a continuous function $f_0: X_f \to \R$ such that
$f=f_0 \circ {\a}_f$ and $fG=(f_0G) \circ {\a}_f$.
By Lemma \ref{l:genQuot} we obtain that $f_0G$ is a Rosenthal family for
$X_f$. Therefore, $f_0 \circ p \in {\mathcal F}(X_f)$ for every $p \in
E(X_f)$. Then also $f_0 \circ gp=f_0g \circ p \in {\mathcal F}(X_f)$
for every $p \in E(X_f)$ and $g \in G$. Now since $f_0G$ separates
points of $X_f$ (Section \ref{s:cycl}), by Lemma
\ref{simple-fr}.3 we conclude that $p: X_f \to X_f$ is a
fragmented map for every $p \in E(X_f)$. This means that $(G,X_f)$
is tame.

(3) $\Leftrightarrow$ (4): Follows from Fact \ref{t:Talagr}.
\end{proof}

\begin{remark} By Rosenthal's dichotomy every bounded sequence
in a Banach space either has a weak Cauchy subsequence or a
subsequence equivalent to the unit vector basis of $l_1$ (the
so-called $l_1$-\emph{sequence}). Recall the definition of
regularity of dynamical systems
originally introduced by
K\"{o}hler \cite{Ko} for cascades in terms of \emph{independent}
sequences.
A compact $G$-space $X$ is \emph{regular} iff for every
$f \in C(X)$ the orbit $fG$ does not contain an $l_1$-sequence (in
other words the second alternative is ruled out in Rosenthal's
dichotomy). By Fact \ref{t:Talagr} it is equivalent to requiring
that $fG$ be a Rosenthal family for $X$ for every $f \in C(X)$. In
fact the notions of regularity and tameness coincide (see
\cite{Gl-tame} (for metrizable systems) and Corollary \ref{c:tame-char} below).
\end{remark}

\begin{cor} \label{c:tame-char}
Let $X$ be a compact (not necessarily metrizable)
$G$-space. The
following conditions are equivalent:
\ben
\item
$(G,X)$ is a tame dynamical system
(that is, $E(X) \subset {\mathcal F}(X,X)$).
\item
$C(X)=Tame(X)$.
\item
$(G,X)$ is regular (in the sense of K\"{o}hler). \een
\end{cor}
\begin{proof}
(2) $\Leftrightarrow$ (3):
Use Proposition \ref{p:tame-f} and Fact \ref{t:Talagr}.

(1) $\Leftrightarrow$ (2): Observe that
by standard arguments (see for example \cite[Prop. 2.9]{GM})
the universal tame $G$-factor of $X$ is naturally isomorphic to $X$ iff $C(X)=Tame(X)$.
\end{proof}

Let $X$ be a compact $G$-space. Then WAP functions on $X$ come
from reflexively representable factors. Similarly, Asplund
functions on a compact $G$-system $X$ are exactly functions which
come from Asplund representable (that is, RN) factors. Every RN
(being HNS) is tame in virtue of Lemma \ref{l:HNS => Tame}. Hence
$$WAP(X) \subset Asp(X) \subset Tame(X)$$ Another way to see these
inclusions for metrizable $X$ is the following proposition
(see also Lemma \ref{l:4} and Section \ref{s:BanHier}).

\begin{prop} \label{p:3f-on}
Let $X$ be a compact metric $G$-space and $f \in C(X)$. Then
\ben
\item
$\cls_p(fG) \subset C(X)$ if and only if $f \in WAP(X)$.
\item
$\cls_p(fG)$ is a (compact) metrizable subspace in $\R^X$ iff $f
\in Asp(X)$ iff $fG$ is a fragmented family of functions on $X$.
\item
$\cls_p(fG) \subset \calB_1(X)={\mathcal F}(X)$ if and only if $f \in
Tame(X)$. \een
\end{prop}
\begin{proof} (1)
Use Grothendieck's theorem: for a compact space $X$,
a bounded subset $A \subset C(X)$ is relatively weakly compact
in $C(X)$ iff it is pointwise relatively compact.

(2) By \cite{GM} we know that $f \in Asp(X)$ iff $fG$ is a
fragmented family of functions on $X$. At the same time Lemma
\ref{l:4} shows that $\cls_p(fG)$ is a (compact) metrizable
subspace in $\R^X$ iff $fG$ is a fragmented family of functions on
$X$.

(3) Apply Proposition \ref{p:tame-f}.
\end{proof}

\begin{remarks} \label{r:old}
\ben
\item
Note that the equivalence ``$f \in Asp(X) \Leftrightarrow$
$\cls_p(fG)$ is metrizable in $\R^X$" is a new characterization of
Asplund functions on metric compact $G$-systems.
\item
For a concrete example of a metric tame system which
is not RN see \cite[Example 14.10]{GM}.
\een
\end{remarks}

Further results concerning tame systems can be found in \cite{Ko},
\cite{Gl-tame}, \cite{G3}, \cite{Huang}, \cite{KL}.

\section{The main results}

\subsection{Banach representations of tame systems}

Let us say that a compact $G$-space $X$ is \emph{Rosenthal
representable} if it admits a faithful representation on a
Rosenthal Banach space
(see Definition \ref{Ros-Banach} above).
Our main result (Theorem \ref{main}) asserts that
a compact metric $G$-space is tame iff
it is Rosenthal representable.
%

\begin{thm} \label{t:1dir}
Every Rosenthal representable (not necessarily metrizable) compact
$G$-space is tame.
In particular, the dynamical system $(Iso(V),B^*)$,
where $B^*$ is the weak$^*$ compact unit ball of $V^*$,
is tame for every Rosenthal Banach space $V$.
\end{thm}

\begin{proof}
 It is enough to show that for every Rosenthal Banach space $V$ the
associated dynamical system $(Iso(V),B^*)$ is tame.
We have to show that $p: B^* \to B^*$ is a fragmented map for every $p \in E(G,B^*)$.

\sk

\nt \textbf{Claim:} For every $p \in E(G,B^*)$ there exists a (uniquely defined)
linear operator $\overline{p}: V^* \to V^*$ with the norm $\leq 1$ such that $p=\overline{p}|_{B^*}$.
\begin{proof}
Consider the enveloping semigroup $E(G,V^*)$ of the separately continuous
action $G \times V^* \to V^*$, where $V^*$ carries the weak$^*$ topology.
That is, as in the case of compact $G$-spaces (Section \ref{s:env}) $E(G,V^*)$
is the pointwise closure $E(G,V^*)$ of the set of all
$g$-translations $\breve{G}=\{\breve{g}: V^* \to V^* \}_{g \in G}$.
Then $E(G,V^*)$ is a compact right topological semigroup. The compactness follows from the fact that
the $G$-orbits in $V^*$ are relatively weak$^*$ compact. Furthermore, every $t \in E(G,V^*)$
(as a pointwise limit of linear isometries) is a linear operator $V^* \to V^*$ with  norm $\leq 1$.
Consider the natural continuous restriction  $\Psi: E(G,V^*) \to E(G,B^*), t \mapsto t|_{B^*}$.
Since $E(G,V^*)$ is compact it easily follows
that $\Psi$ is onto. In fact $\Psi$ is even injective (hence, a homeomorphism) because every linear operator
$V^* \to V^*$  is defined uniquely by its restriction to $B^*$.
\end{proof}

Let $p \in E(G,B^*)$. According to the claim $p=\overline{p}|_{B^*}$,
where $\overline{p}$ is a linear operator $\overline{p}: V^* \to V^*$ with norm $\leq 1$.
Then, for every vector $f \in V$, the composition $f \circ \overline{p}: V^* \to \R$
is a linear bounded (hence norm continuous) functional on
$V^*$. That is, $f \circ \overline{p} \in V^{**}$ belongs to the second dual.
By the reformulation of a theorem of E. Saab and P. Saab
(see Fact \ref{t:SS} and Remark \ref{r:WRN}.1) mentioned above,
the corresponding restriction $(f \circ \overline{p}) |_{B^*}: B^* \to \R$ is a fragmented function.
Next note that $V$ separates points of $B^*$. Since
$f \circ p=(f \circ \overline{p})|_{B^*}$ is fragmented for every $f \in V$, we can apply Lemma
\ref{simple-fr}.3. It follows that $p: {B^*} \to {B^*}$ is fragmented, completing the proof of the theorem.
\end{proof}

\begin{lem} \label{l:system}
Let $X$ be a compact metric $G$-space.
The following conditions are equivalent:
\begin{enumerate}
\item
$(G,X)$ is representable on a separable Rosenthal Banach
space.
\item
$(G,X)$ admits countably many representations
on separable Rosenthal Banach spaces
which separate points of $X$.
\end{enumerate}
\end{lem}
\begin{proof}
Observe that the $l_2$-sum of a sequence of
separable Rosenthal
Banach spaces is again a separable Rosenthal space. Indeed,
this follows for instance from the
equivalence of (1) and (2) in Odell and Rosenthal's Theorem (Fact
\ref{t:Ros}).
The rest is similar to the proof of \cite[Lemma 4.9]{Menz}.
\end{proof}


\begin{thm} \label{t:general}
Let $X$ be a compact $G$-space, $F \subset C(X)$ a Rosenthal
family for $X$ such that
$F$ is $G$-invariant (that is, $fG \subset F \ \ \forall f \in F$). Then
\ben
\item
there exist: a Rosenthal Banach space $V$, an injective mapping
$\nu: F \to B_V$ into the unit ball $B_V$ and a continuous representation
$$
h: G \to Iso(V), \ \ \ \a:
X \to V^*
$$
of $(G,X)$ on $V$ such that $\a$ is
a weak$^*$ continuous map (topological embedding if $F$
separates points of $X$) and
$$
f(x)= \langle \nu(f), \a(x)
\rangle \ \ \ \forall \ f \in F \ \ \forall \ x \in X.
$$
Thus the following diagram commutes

\begin{equation*}
\xymatrix{ F \ar@<-2ex>[d]_{\nu} \times X
\ar@<2ex>[d]^{\a} \ar[r]  & \R \ar[d]^{id_{\R}} \\
V \times V^* \ar[r]  &  \R }
\end{equation*}

\item
If $X$ is metrizable then in addition we can suppose that
$V$ is separable and there exists a homeomorphic embedding
$$\nu_0: K:=\cls_p(F) \hookrightarrow V^{**}$$ furnishing
$V^{**}$ with its weak$^*$ topology
and the following diagram commutes

\begin{equation*}
\xymatrix{ K \ar@<-2ex>[d]_{\nu_0} \times X
\ar@<2ex>[d]^{\a} \ar[r]  & \R \ar[d]^{id_{\R}} \\
V^{**} \times V^* \ar[r]  &  \R }
\end{equation*}

\een
\end{thm}
\sk
\begin{proof} \textbf{(1)}  \ Let $W$ be the symmetrized convex hull of $F$; that is,
$$W:={\co}(F \cup -F).$$

\sk

\nt \textbf{Claim 1:} $W$ is also a Rosenthal family for $X$.
\begin{proof} It is easy to see that $F \cup -F$ is a
Rosenthal family for $X$. Now apply Proposition \ref{c:conv}.
\end{proof}

For brevity of notation let $\Acal := C(X)$ denote the Banach
space $C(X)$, $B$ will denote its unit ball,
and $B^*$ will denote
the weak$^*$ compact unit ball of the dual space $\Acal^*
=C(X)^*$.

\sk

\nt \textbf{Claim 2:} \emph{$W$ is a Rosenthal family for
$B^*$.}

\begin{proof}
Apply Proposition \ref{p:from X to B*}.
\end{proof}

\br

Consider the sequence of sets $M_n:=2^n W + 2^{-n} B$.
Since $W$ is convex and
symmetric, we can apply the construction of
Davis-Figiel-Johnson-Pelczy\'nski
\cite{DFJP}
as follows.
Let $\| \ \|_n$ be the Minkowski
functional of the set $M_n$. That is,
$$
\| v\|_n = \inf\ \{\lambda
> 0 \bigm| v\in \lambda M_n\}.
$$
Then $\| \ \|_n$ is a norm on $\Acal$ equivalent to the given norm
of $\Acal$. For $v\in \Acal,$ set
$$
N(v):=\left(\sum^\infty_{n=1} \| v \|^2_n\right)^{1/2} \hskip
0.1cm \text{and} \hskip 0.1cm \hskip 0.1cm V: = \{ v \in \Acal
\bigm| N(v) < \infty \}.
$$
Denote by $j: V \hookrightarrow \Acal$ the inclusion map. Then
$(V,N)$ is a Banach space, $j: V \to \Acal$ is a continuous linear
injection and
$$
W \subset j(B_V)=B_V.
$$
Indeed, if $v \in W$ then $2^nv \in M_n$, hence $\| v\|_n \leq
2^{-n}$ and $N(v)^2 \leq \sum_{n \in \N} 2^{-2n} \allowbreak <1.$

\sk

\sk
The given action $G \times X \to X$ induces the natural linear norm preserving
continuous right action $C(X) \times G \to C(X)$ on the Banach space $\Acal=C(X)$.
It follows by the construction that $W$ and $B$ are $G$-invariant subsets in $\Acal$.
This implies that $V$ is a $G$-invariant subset of $\Acal$ and the restricted natural linear action
$V \times G \to V, \ \ (v,g) \mapsto vg$ is norm preserving, that is, $N(vg)=N(v)$.
Moreover, by the definition of the norm $N$, we can show that this action is norm
continuous (use the fact that, for each $n \in \N$, the norm
$\norm{\cdot}_n$ on $\Acal$ is equivalent to the given norm on
$\Acal$). Therefore, the co-homomorphism $h: G \to {\Iso}(V), \
h(g)(v):=vg$ is well defined and continuous.

Let $j^*: \Acal^* \to V^*$ be the adjoint map of $j: V \to \Acal$.
Define $\a: X \to V^*$ as follows. For every $x \in X \subset
C(X)^*$ set $\a(x)=j^*(x)$. Then $(h,\a)$ is a continuous
representation of $(G,X)$ on the Banach space $V$.


\sk

By the construction $F \subset W \subset B_V$.
Define $\nu: F \hookrightarrow B_V$ as the natural inclusion.
Then
\begin{equation} \label{F}
f(x)= \langle \nu(f), \a(x) \rangle \ \ \ \forall \ f \in F \ \ \forall \ x \in X.
\end{equation}

It follows in particular that if $F$ separates points of $X$ then $\a$ is an embedding.

\br

\nt \textbf{Claim 3:} $j(B_V) \subset \bigcap_{n \in \N} M_n =
\bigcap_{n \in \N} (2^n W + 2^{-n}B)$.
\begin{proof}
The norms $\norm{\cdot}_n$ on $\Acal$ are equivalent to each
other. It follows that if $v \in B_V$ then $\| v\|_n < 1$ for all
$n \in \N$. That is, for every $n \in \N$, $v \in \lambda_n M_n$
for some $0 < \lambda_n <1$. By the construction $M_n$ is a convex
subset containing the origin. This implies that $\lambda_n M_n
\subset M_n$. Hence $j(v)=v \in M_n$ for every $n \in\N$.
\end{proof}

\sk

\nt \textbf{Claim 4:} The set $\bigcap_{n \in \N} M_n$ (and hence
also its subset $j(B_V)$) is sequentially precompact in the second
dual $(\Acal^{**}, \sigma(\Acal^{**}, \Acal^*))$ (i.e.
$\Acal^{**}$ endowed with its weak$^*$ topology).

\begin{proof} We use the argument of
\cite[Lemma 1 (xii), p. 323]{DFJP} with some minor changes.

Let $\{c_n\}_1^{\infty}$ be a sequence in $\bigcap_{k \in \N}
M_k$. Then for each fixed $n \in \N$ and every $k \in \N$ we can
represent $c_n$ as
\begin{equation} \label{nk}
c_n = 2^k w_n^k +2^{-k} b_n^k
\end{equation}
with $w_n^k \in W, \ b_n^k \in B$. By Claim 2 we know that
$W$ is a Rosenthal family for $B^*$.
Thus by Fact \ref{t:Talagr}, $W$ is sequentially precompact in $\R^{B^*}$,
hence also
in $\R^{C(X)^*}$. It follows that
$W$ is sequentially precompact in $(\Acal^{**}, \sigma(\Acal^{**}, \Acal^*))$
(as $\Acal=C(X)$).
Applying a diagonal process we can
choose a subsequence $\{n_i\}_{i=1}^{\infty}$ such that for each
$k$ the sequence $\{w_{n_i}^k\}_{i=1}^{\infty}$ is
$\sigma(\Acal^{**}, \Acal^*)$-convergent to an element, say,
$x_k^{**} \in \Acal^{**}$. In order to simplify our notation we
will relabel our sequences and now assume that for every $k$,
\begin{equation}\label{n-limit}
w^*{\text{-}}\lim_{n} w_n^k = x_k^{**}
\end{equation}
in $\Acal^{**}$.

\sk

\nt \textbf{Claim 5:} The sequence $2^k x_k^{**}$ is norm Cauchy
in the second dual $\Acal^{**}$.

\begin{proof}
Note first that by (\ref{nk}) we have:


\begin{equation}\label{norm}
\|2^k w_n^k - 2^l w_n^l\| =
\|2^{-k} b_n^k - 2^{-l} b_n^l\| \le 2\cdot
2^{-\min\{k,l\}}.
\end{equation}
Now for a fixed unit vector $x^* \in \Acal^*$ we have:
\begin{gather*}
|\langle x^*, 2^k x_k^{**} - 2^l x_l^{**} \rangle |  \le\\
|\langle x^*, 2^k x_k^{**} - 2^k w_n^k \rangle |  + |\langle x^*,
2^k w_n^k - 2^l w_n^l \rangle \rangle | +
|\langle x^*,  2^l w_n^l - 2^k x_l^{**} \rangle |  \le\\
|\langle x^*, 2^k x_k^{**} - 2^k w_n^k \rangle |  + 2\cdot
2^{-\min\{k,l\}} + |\langle x^*,  2^l w_n^l - 2^k x_l^{**} \rangle|.
\end{gather*}

By (\ref{n-limit})
and (\ref{norm})
 we get
\begin{equation*}
|\langle x^*, 2^k x_k^{**} - 2^l x_l^{**} \rangle |  \le
2\cdot 2^{-\min\{k,l\}},
\end{equation*}
hence
$$
||2^k x_k^{**} - 2^l x_l^{**}|| =
\sup_{||x^*|| = 1} |\langle x^*, 2^k x_k^{**} - 2^l x_l^{**}
\rangle |
\le 2\cdot 2^{-\min\{k,l\}}.
$$
\end{proof}

\sk

We let $\xi :=\lim_{k \to \infty} 2^k x_k^{**}$ in the Banach
space $\Acal^{**}$.

\br

\nt \textbf{Claim 6:} $ w^*{{\text-}}\lim_{n \to \infty} c_n=\xi$;
i.e. $\{c_n\}_{n=1}^{\infty}$ is $\sigma(\Acal^{**},
\Acal^*)$-convergent to $\xi \in \Acal^{**}$.
\begin{proof}
Again fix a unit vector $x^* \in \Acal^*$.
Given $\ep >0$ fix $k_0 \in \N$ such that $2^{-k_0} < \ep$, and
$$||2^k x_k^{**} - \xi || < \ep$$ for $k > k_0$. By (\ref{nk}),
$\|c_n - 2^k w_n^k\| \le 2^{-k}$, hence for $k
> k_0$ and every $n \in \N$ we get
\begin{gather*}
|\langle x^*, c_n - \xi \rangle |  \le \\
|\langle x^*, c_n - 2^k w_n^k \rangle | +
|\langle x^*, 2^k w_n^k - 2^k x_k^{**} \rangle| +
|\langle x^*, 2^k x_k^{**} - \xi \rangle| \le\\
 \ep + |\langle x^*, 2^k
w_n^k - 2^k x_k^{**} \rangle|  + \ep.
\end{gather*}
For an arbitrary but fixed $k > k_0$, by (\ref{n-limit})
again, there is an $n_0 \in \N$ such that for $n > n_0$,
$$
|\langle x^*, 2^k w_n^k - 2^k x_k^{**} \rangle| < \ep.
$$
Combining these inequalities we see that for $n > n_0$
$$
|\langle x^*, c_n - \xi \rangle |  \le 3\ep.$$
\end{proof}
Note that now we have completed the proof of Claim 4.
\end{proof}

\sk

Let $\overline{B_V}$ be the $\sigma(V^{**}, V^*)$-closure of the
unit ball $B_V$ in $V^{**}$ (in fact $\overline{B_V}$ is the unit
ball of the second dual $V^{**}$ by
Goldstine's theorem). Similarly,
$\overline{j^{**}(B_V)}=\overline{j(B_V)}$ denotes the
$\sigma(\Acal^{**}, \Acal^*)$-closure in the second dual
$\Acal^{**}$.
From Claim 4 we know that $j(B_V)$ is sequentially precompact in
$\overline{j(B_V)}$.
In order to show
that $V$ is a Rosenthal space we have to establish that
$B_V$ is a Rosenthal family for $B^*$ (Lemma \ref{l:crit}).
In turn, this is equivalent by Fact \ref{t:Talagr} to showing that $B_V$
is sequentially precompact in $\R^{B^*}$, or equivalently
in the $\sigma(V^{**}, V^*)$-compact space $\overline{B_V}$.

It is enough to find a homeomorphism
$\eta: \overline{B_V} \to
\overline{j(B_V)}$ such that $\eta(b)=j(b)$ for every $b \in B_V$.
Consider the canonical second adjoint map
$$
j^{**}: V^{**} \to \Acal^{**}
$$
This map is injective by \cite[Lemma 1(iii)]{DFJP}. Using the
compactness of $j^{**}(\overline{B_V})$ we get
$j^{**}(\overline{B_V})=\overline{j^{**}(B_V)}=
\overline{j(B_V)}$. Now the required homeomorphism $\eta$ is the
restriction of $j^{**}$ to the $\sigma(V^{**}, V^*)$-compact space
$\overline{B_V}$.

We obtain that $V$ is a Rosenthal space and the assertion \textbf{(1)} is proved.


\br

\nt \textbf{(2)} \ If the compact space $X$ is metrizable then $C(X)$ is
separable and it is also easy to see that $(V, N)$ is separable.
The pointwise closure
$K:=\cls_p (F)$ of $F$ in $\R^X$ is a subset of $b \calB_1(X)$.
Consider now the mapping $T: b \calB_1(X)  \to C(X)^{**}$ (see
Remark \ref{r:functionals}). This map induces by Proposition
\ref{p:B-star} a homeomorphic embedding
$$
T: \ K=\cls_p(F) \to \cls_p(T(F)) \subset b \calB_1(B^*) \cap
C(X)^{**}
$$
of $K$ into $\Acal^{**}=C(X)^{**}$.

By construction $F$ is a subset of $V$. Consider the second
adjoint map
$$
j^{**}: V^{**} \to \Acal^{**}.
$$
Recall again that
this map is injective by \cite[Lemma 1 (iii)]{DFJP}. Therefore
$j^{**}$ induces a homeomorphism between the weak$^*$ compact spaces
$\cls_{V^{**}}(F)$ and $\cls_{\Acal^{**}}(F)$. Summing up we can
define the desired homeomorphic embedding as follows

$$
\nu_0: K
\to V^{**}, \ \ \ f \mapsto
({j^{**}})^{-1}(T(f)).
$$
Observe that $\nu_0(f)=\nu(f)$ for every $f \in F$. Since $F=\nu(F)$ is pointwise dense in $K$, using \ref{F}, we get
$$
f(x)= \langle \nu_0(f), \a(x)
\rangle \ \ \ \forall \ f \in K ,\ \ \forall \ x \in X.
$$
So, Theorem \ref{t:general} is proved.
\end{proof}

\sk

Recall again (see Section \ref{subs:RN}) that a compact space $X$
is called Radon-Nikod\'ym (RN) if $X$ is homeomorphic to a
weak$^*$ compact subset of the dual $V^*$ for an Asplund space
$V$. A well known result characterizes Rosenthal spaces as those
Banach spaces whose dual has the weak Radon-Nikod\'ym property
\cite[Corollary 7.3.8]{Tal}. It is therefore natural to introduce
the following definition.

\begin{defn} \label{d:WRN}
We say that a compact topological space $X$ is \emph{weakly
Radon-Nikod\'ym} (WRN) if $X$ is homeomorphic to a weak$^*$
compact subset of the dual $V^*$ of a Rosenthal space $V$. This
definition agrees with the definition of WRN subsets (cf. Remark
\ref{r:WRN}.2).
A compact $G$-space $X$ is called
a WRN $G$-space if $X$, as a $G$-space, is Rosenthal representable.
\end{defn}

As a corollary of Theorem \ref{t:general} we get the
following characterization of WRN compacta.

\begin{thm} \label{t:WRN}
Let $X$ be a compact space. The following conditions
are equivalent: \ben
\item
$X$ is WRN.
\item
There exists a Rosenthal family $F \subset C(X)$ of $X$
which separates the points of $X$.
 \een
\end{thm}

\begin{remark}
We mention without proof that a compact $G$-space $X$ is RN (that
is, Asplund representable) if and only if there exists a
bounded
subset $F \subset C(X)$ such that $F$ separates points of $X$, $F$ is
$G$-invariant and $F$ is a fragmented family of functions
(Definition \ref{d:fr-family}). Comparing this with Theorem
\ref{t:X-repres-iff} below we see that there exists a
complete analogy between RN-systems and fragmented families
on one side and WRN-systems and Rosenthal families on the
other.
\end{remark}


Next we present a characterization of tame functions in terms of
Banach representations.

\begin{thm} \label{t:functrepr}
Let $X$ be a compact $G$-space.
The following conditions are equivalent:
\begin{enumerate}
\item $f: X \to \R$ is tame.
\item $f: X \to \R$ comes from a Rosenthal Banach
space. That is, there exist a continuous representation $(h,\a)$
of $(G,X)$ on a Rosenthal Banach space $V$ and a vector $v \in V$
such that

$$f(x)= \langle v, \a(x)
\rangle \ \ \  \forall \ x \in X.$$
\end{enumerate}

If $X$ is metrizable we can suppose in addition in (2) that $V$ is
separable.

\end{thm}
\begin{proof} (2) $\Rightarrow$ (1):
Directly follows by Theorem \ref{t:1dir}.

(1) $\Rightarrow$ (2): Let $f \in Tame(X)$. This means 
By Proposition \ref{p:tame-f} that the orbit $fG$ is a Rosenthal
family for $X$. Now we can apply Theorem \ref{t:general} to the
family $F:=fG$.
\end{proof}

Here are the promised Banach space
characterizations of tame and metrizable tame systems.

\begin{thm} \label{t:tame=approx}
Let $X$ be a compact $G$-space.
The following conditions are equivalent:
\begin{enumerate}
\item $(G,X)$ is a tame $G$-system.
\item $(G,X)$ is Rosenthal approximable.
\end{enumerate}
\end{thm}
\begin{proof} (2) $\Rightarrow$ (1):
Apply again Theorem \ref{t:1dir} and take into account that the class
of tame $G$-systems is closed under subsystems and arbitrary
products for every given $G$ (Lemma \ref{l:tame-prop}).

(1) $\Rightarrow$ (2): First of all note that $C(X)=Tame(X)$ by
Corollary \ref{c:tame-char}. Applying Theorem \ref{t:functrepr} we
conclude that every $f \in C(X)=Tame(X)$ on a compact $G$-space
$X$ comes from a Rosenthal representation. Continuous functions
separate points of $X$. This implies that there exist sufficiently
many Rosenthal representations of $(G,X)$.
\end{proof}

\begin{thm} \label{main}
Let $X$ be a compact metric $G$-space. The following conditions
are equivalent:
\begin{enumerate}
\item
$(G,X)$ is tame.
\item
$(G,X)$ is representable on a separable Rosenthal Banach
space.
\end{enumerate}
\end{thm}
\begin{proof} (2) $\Rightarrow$ (1):
Apply once again Theorem \ref{t:1dir}.

(1) $\Rightarrow$ (2):
Since the compact space $X$
is metrizable there is a sequence of functions $f_n \in
C(X)=Tame(X)$ which
separates
points of $X$. For each $f_n$ we can
construct by Theorem \ref{t:functrepr} a continuous Rosenthal
representation $(h_n,\a_n)$ of $(G,X)$ such that our original function
$f_n$ comes from the system  $(G,\a_n(X))$. Applying Lemma
\ref{l:system} we conclude that $(G,X)$ is Rosenthal
representable.
\end{proof}

If $X$ is a tame, not necessarily metrizable, dynamical
$G$-system then the induced systems $(G, B^*)$
(on the weak$^*$ compact unit ball $B^*$ of $C(X)^*$) and $(G, P(X))$
(where $P(X)$ denotes the weak$^*$ compact subspace of $B^*$ consisting of all
probability measures on $X$) are tame as well (see Corollary \ref{t:B(X)tame}).  For metrizable $X$
this is \cite[Theorem 1.5]{Gl-tame}.
In fact one may show a stronger result:

\begin{thm} \label{t:X-repres-iff}
Let $X$ be a compact $G$-space.
The following conditions are equivalent:
\ben
\item
$(G,X)$ is Rosenthal representable
(i.e. $X$ is a WRN \ $G$-space).
\item
There exists a $G$-invariant Rosenthal family
$A \subset C(X)$ for $X$
which separates points of $X$.
\item
$(G,B^*)$ is Rosenthal representable.
\item
$(G,P(X))$ is Rosenthal representable.
\een
\end{thm}
\begin{proof}
(1) $\Leftrightarrow$ (2):
Apply Theorem \ref{t:general} and Lemma \ref{l:crit}.

(2) $\Rightarrow$ (3):
Let $A$ be a $G$-invariant point separating Rosenthal family for
$X$. We will show that there exists a $G$-invariant point separating Rosenthal family for
$B^*$.
Produce inductively the sequence $A_n:= A_1 \cdot A_1 \cdots A_1$
($n$-times),
where $A_1:=A$. We can suppose that
$A \subset B_{C(X)}$ and contains the constant function $1$.
We can show by diagonal arguments
(use Fact \ref{t:Talagr} or Prop \ref{p:aft}) that $A_n$ is also a Rosenthal family for $X$. Furthermore (by the same results)
it is easy to show that the $G$-invariant family $M:=\cup_n
2^{-n}A_n$ is Rosenthal for $X$.
By Stone-Weierstrass theorem the algebra $span(M)$
(linear span of $M$ in $C(X)$) is dense in $C(X)$. This implies
that $M$ itself separates the points of $B^*$.
By Proposition \ref{p:from X to
B*}, $M$ is a ($G$-invariant) Rosenthal family for $B^*$.
Now we apply the part (1) $\Leftrightarrow$ (2) to the case of $(G,B^*)$.

(3) $\Rightarrow$ (4) and (4) $\Rightarrow$ (1) are obvious.
\end{proof}


\begin{thm} \label{t:B(X)tame}
Let $X$ be a compact $G$-space. Then
$(G,X)$ is tame iff $(G,B^*)$ (equivalently, $(G,P(X))$) is tame.
\end{thm}
\begin{proof}
It is enough to show that $(G,B^*)$ is tame for every tame system $(G,X)$.
By Theorem \ref{t:X-repres-iff} we know that $X$ is $G$-embedded into a $G$-product $\prod_{i \in I} X_i$ of
Rosenthal representable $G$-systems $X_i$. Corresponding $G$-system $B_i^*$ is Rosenthal representable, hence tame
by virtue of Theorems \ref{t:tame=approx} and \ref{t:1dir}. Now observe that
$B^*$ is naturally $G$-embedded into the $G$-product $\prod_{i \in I} B_i^*$
%
and use Lemma \ref{l:tame-prop}.
\end{proof}


\begin{thm} \label{t:mat}
Let $G$ be a topological group. The following conditions are equivalent:
\ben
\item $f: G \to \R$ is tame.
\item $f$ is a matrix coefficient of a continuous
co-representation of $G$ on a Rosenthal space.
That is, there exist: a Rosenthal space $V$, a continuous
co-homomorphism $h: G \to Iso(V)$, vectors $v \in V$ and $\psi
\in V^*$ such that $f(g)=\psi(vg)$ for every $g \in G$.
\een
\end{thm}
\begin{proof} (1) $\Rightarrow$ (2):
Choose a tame $G$-compactification $\nu: G \to X$ of $G$
and a continuous function $f_0: X \to \R$ such that $f=f_0 \circ \nu$.
Now we can apply Theorem \ref{t:functrepr} to $f_0$ getting the
desired $V$ and vectors $v$ and $\psi:=\a(\nu(e))$.

(2) $\Rightarrow$ (1): Apply Theorem \ref{t:1dir}.
\end{proof}


\sk
\subsection{Compact spaces in the second dual of Rosenthal spaces}

Again we remind the reader that
a compact topological space $K$ is {\it Rosenthal}
if it is homeomorphic to a pointwise compact
subset of the space $\calB_1(X)$ of functions of the first Baire
class on a Polish space $X$.

\begin{defn} \label{d:RosTypes}
\ben
\item
We say that a compact space $K$ is \emph{strongly Rosenthal} if
$K$ is a subspace of $\calB_1(X)$ with compact metrizable $X$.
\item
We say that a compact space $K$ is an \emph{admissible Rosenthal
compactum}
(or simply \emph{admissible}) if there exists a
compact metric space $X$ and a
Rosenthal family $F$ for $X$ such that $K \subset \cls_p(F)$.

\een
\end{defn}

In the second definition it follows that $K \subset \cls_p(F) \subset {\mathcal B}_1(X)$.
Hence every admissible compactum is a strongly Rosenthal
compactum.
Clearly every strongly Rosenthal compact space is Rosenthal.

Every subset $F \subset C(X)$ is norm separable for
a compact metric $X$.
Hence such an $F$ is also separable with respect to the pointwise
convergence topology. Thus
in Definition
\ref{d:RosTypes}.2, we can assume that $F$ is countable.

%
%




Pol's example, mentioned in the introduction, shows that
not every separable Rosenthal compactum is strongly Rosenthal
(and a fortiori also not admissible).

\begin{lem} \label{l:prod}
The classes of Rosenthal, strongly Rosenthal and admissible
compact spaces are closed under the operations of
passing to closed subspaces and taking countable products.
\end{lem}
\begin{proof}
The hereditarily property of each of these classes is obvious.
In order to see that the countable product $K:=\prod_n K_n$ of
Rosenthal compact spaces $K_n$ is again Rosenthal we consider the
topological (disjoint) sum $X:= \coprod_{n \in \N} X_n$,
where $X_n$ is a Polish space for which $K_n \subset
\calB_1(X_n)$. Then $K$ can be embedded into $\calB_1(X)$ as
follows.
For each element
$$
f:=(f_1, f_2, \cdots ) \in \prod K_n = K
$$
there exists a uniquely defined function $j(f): X \to \R$ such
that the restriction of $j(f)$ on $X_n$ is exactly $f_n$. Clearly,
$j(f)$ is a Baire 1 function on $X$.
This defines the continuous map $j: K \to \calB_1(X)$.
Since $j$ is injective and $K$ is compact we conclude that
$j$ is a topological embedding.

Suppose now that each $K_n$ is strongly Rosenthal. Then,
by definition, we can
assume in addition that each $X_n$ as above
is a compact metric space. Now
it is easy to see that $K$ admits a topological embedding into
$\calB_1(X^*)$, where $X^*:=X \cup \{\infty\}$ is the one point
compactification of $X=\coprod_{n \in \N} X_n$. In this case we
define $j^*: K \to \calB_1 (X^*)$ by $j^*(f)(\infty)=0$ and
$j^*(f)(x)=j(f)(x)$ for every $x \in X$. Then again $j^*$ is well
defined and it embeds $K$ into $\calB_1 (X^*)$.

Finally we consider the case where each $K_n$ is admissible.
As in the second case we have the topological embedding
$$
j^*: K \hookrightarrow \calB_1 (X^*)
$$
We have to show that there exists a family $F \subset C(X^*)$ such
that $j^*(K) \subset \cls_p(F)$. For each $n \in \N$ fix a
countable subset $F_n \subset C(X_n)$ such that $K_n \subset
\cls_p (F_n) $. It is enough to show our assertion in the case
where $\cls_p (F_n) = K_n$. For each $k \in \N$ consider the
elements of the type
$$
f:=(f_1, f_2, \cdots, f_k, 0_{k+1}, 0_{k+2}, \cdots ) \in \prod
K_n = K,
$$
where $f_i \in F_i$ for every $i \leq k$ and each $0_{k+m}$
denotes
the constant zero function on $X_{k+m}$
(again without restriction of generality
we can assume that $0_{t} \in F_t$ for every $t \in \N$).
Varying $k \in \N$ and $f_i \in F_i$
with $i \leq k$ we get a countable subset $F_0 \subset K$. Clearly
this subset is dense in the product space $K=\prod_{n \in \N} K_n$.
It is easy to see that its image $F:=j^*(F_0)$ is the
required family. That is, $F \subset C(X^*)$ and $j^*(K) \subset
\cls_p(F)$.
\end{proof}

\begin{prop} \label{p:EisTame}
Let $X$ be a compact metric $G$-space.
The following conditions are equivalent:
\ben
\item
$(G,X)$ is a tame system.
\item
$E(X)$ is an
admissible Rosenthal compactum.
\item
$E(X)$ is a Rosenthal compactum.
\een
\end{prop}
\begin{proof}

(1) $\Rightarrow$ (2):  Let $(G,X)$ be a tame system. Then every
continuous function $f \in C(X)$ is tame. This means that $fG$ is
a Rosenthal family for $X$. Then the compact space
$E^f:=\cls_p(fG)$ is a subset of $\calB_1(X)$ (Proposition
\ref{p:tame-f} and Corollary \ref{c:fr=B1}.2).
%
So the compactum $E^f$ is admissible.
Since $X$ is a metrizable compact space one may choose a countable
set of functions $\{f_m\}_{m \in \N}$ which separates the points
in $X$. Then $E(X)$ can be naturally embedded into the countable
product $K:=\prod_m E^{f_m}$ which is admissible by Lemma
\ref{l:prod}.

(2) $\Rightarrow$ (3): is clear.

(1) $\Leftrightarrow$ (3): Follows directly from the definitions.
\end{proof}

We have the following related purely topological result:

\begin{thm} \label{t:tame-comp} Let $K$ be a compact space.
The following conditions are equivalent:
\ben
\item
$K$ is an admissible
Rosenthal
compactum.
\item  $K$ is
homeomorphic to
a weak$^*$ closed
subset in the second dual of a separable Rosenthal Banach space $V$.
\een
\end{thm}
\begin{proof}
(1) $\Rightarrow$ (2):
Let $K$ be an admissible compactum.
By Definition \ref{d:RosTypes}
there exists a compact metric space $X$ and a Rosenthal
family $F \subset C(X)$ such that $K \subset
\cls_p(F) \subset {\mathcal B}_1(X)$. We have to show that $K$ is
homeomorphic to a weak$^*$ closed subset in the second
dual of a separable Rosenthal Banach space $V$. It is enough to
establish this for the case of $K=\cls_p(F)$. But this fact
follows directly from Theorem \ref{t:general}.2
when one considers a trivial (identity) $G$-action on $X$.

(2) $\Rightarrow$ (1):
 We have to show that $K$ is an admissible
compactum. It is enough to show this for the particular case where
$K:=B^{**}:=B_{V^{**}}$, the unit ball in the second dual. Since
$V$ is separable, $X:=B^*$, the weak$^*$ compact unit ball in
$V^*$ is a metrizable compact space. By our assumption $V$ is a
Rosenthal space. Then by Fact \ref{t:Ros}.3, $K=B^{**}$ is
naturally embedded into $\calB_1(X)$ with $X:=B^*$. By Goldstine's
theorem the unit ball $B:=B_V$ of $V$ is weak$^*$-dense in
$B^{**}$. At the same time $B$ can be treated as a (bounded)
subset of $C(X)$. Thus, $B$ is a Rosenthal family for $X$. Hence
the compactum $K$ is admissible in the sense of Definition
\ref{d:RosTypes}.2.
%
\end{proof}



\bibliographystyle{amsplain}

\end{document}